\documentclass{amsart}
\usepackage[latin1]{inputenc}
\usepackage{amsmath}
\usepackage{amssymb}
\usepackage[all]{xy}
\usepackage{mathrsfs}
\usepackage{amsthm}
\usepackage{graphicx}
\usepackage{stmaryrd}

\theoremstyle{plain}
\newtheorem{thm}[equation]{Theorem}
\newtheorem{lem}[equation]{Lemma}
\newtheorem{prop}[equation]{Proposition}
\newtheorem{cor}[equation]{Corollary}

\theoremstyle{definition}
\newtheorem{defin}[equation]{Definition}

\theoremstyle{remark}
\newtheorem{remark}[equation]{Remark}

\numberwithin{equation}{subsection}

\def\sheafEnd{\mathcal{E} \hspace{-1pt} \mathit{nd}}
\def\sheafHom{\mathcal{H} \hspace{-1pt} \mathit{om}}

\def\sheafEnd{\mathcal{E} \hspace{-1pt} \mathit{nd}}
\def\sheafHom{\mathcal{H} \hspace{-1pt} \mathit{om}}

\newcommand{\Z}{\mathbb{Z}}
\newcommand{\calA}{\mathcal{A}}
\newcommand{\calB}{\mathcal{B}}
\newcommand{\calC}{\mathcal{C}}
\newcommand{\calD}{\mathcal{D}}
\newcommand{\calE}{\mathcal{E}}
\newcommand{\calF}{\mathcal{F}}

\newcommand{\calH}{\mathcal{H}}
\newcommand{\calI}{\mathcal{I}}
\newcommand{\calJ}{\mathcal{J}}
\newcommand{\calK}{\mathcal{K}}
\newcommand{\calL}{\mathcal{L}}
\newcommand{\calM}{\mathcal{M}}
\newcommand{\calN}{\mathcal{N}}
\newcommand{\calO}{\mathcal{O}}
\newcommand{\calP}{\mathcal{P}}

\newcommand{\calR}{\mathcal{R}}
\newcommand{\calS}{\mathcal{S}}
\newcommand{\calT}{\mathcal{T}}

\newcommand{\calV}{\mathcal{V}}

\newcommand{\calX}{\mathcal{X}}
\newcommand{\calY}{\mathcal{Y}}
\newcommand{\calZ}{\mathcal{Z}}

\newcommand{\lotimes}{{\stackrel{_L}{\otimes}}}
\newcommand{\rcap}{{\stackrel{_R}{\cap}}}
\newcommand{\Gm}{{\mathbb{G}}_{\mathbf{m}}}

\newcommand{\For}{{\rm For}}
\newcommand{\Id}{{\rm Id}}

\newcommand{\Coh}{\mathsf{Coh}}
\newcommand{\QCoh}{\mathsf{QCoh}}
\newcommand{\Sh}{\mathsf{Sh}}
\newcommand{\Mod}{\mathsf{-Mod}}

\newcommand{\qis}{{\rm qis}}

\newcommand{\fg}{{\rm fg}}

\newcommand{\op}{{\rm op}}

\newcommand{\rmi}{\mathrm{(i)}}
\newcommand{\rmii}{\mathrm{(ii)}}

\newcommand{\scra}{\mathscr{A}}
\newcommand{\scrb}{\mathscr{B}}

\newcommand{\Sym}{{\rm Sym}}

\newcommand{\Kth}{\mathsf{K}}

\newcommand{\bc}{\mathrm{bc}}

\begin{document}

\title[Linear Koszul duality II]{Linear Koszul duality II -- Coherent sheaves \\ on perfect sheaves}

\author{Ivan Mirkovi\'c}
\address{University of Massachusetts, Amherst, MA, USA.}
\email{mirkovic@math.umass.edu}

\author{Simon Riche}
\address{Universit{\'e} Blaise Pascal - Clermont-Ferrand II, Laboratoire de Math{\'e}matiques, CNRS, UMR 6620, Campus universitaire des C{\'e}zeaux, F-63177 Aubi{\`e}re Cedex, France}
\email{simon.riche@math.univ-bpclermont.fr}

\thanks{I.M.~was supported by NSF grants. S.R.~was supported by ANR Grants No.~ANR-09-JCJC-0102-01, ANR-10-BLAN-0110 and ANR-13-BS01-0001-01.}

\begin{abstract}
In this paper we continue the study (initiated in a previous paper) of linear Koszul duality, a geometric version of the standard duality between modules over symmetric and exterior algebras. We construct this duality in a very general setting, and prove its compatibility with morphisms of vector bundles and base change.
\end{abstract}

\maketitle

\section*{Introduction}

\subsection{Linear Koszul duality}

In \cite{MR} we have defined and initiated the study of \emph{linear Koszul duality}, a geometric version of the standard Koszul duality between (graded) modules over the symmetric algebra of a vector space $V$ and (graded) modules over the exterior algebra of the dual vector space $V^*$ (see e.g.~\cite{BGG,GKM}). In our setting of \cite{MR}, we replaced the vector space $V$ by a $2$-term complex of locally free sheaves on a (smooth) scheme $X$, and $V^*$ by a shift of the dual complex, and obtained an equivalence of derived categories of dg-modules over the symmetric algebras (in the \emph{graded} sense) of these complexes. As an application, given a vector bundle $E$ over $X$ and subbundles $F_1, F_2 \subset E$, for an appropriate choice of a $2$-term complex we obtained an equivalence of derived categories of (dg-)sheaves on the derived intersections $F_1 \, \rcap_E \, F_2$ and $F_1^\bot \, \rcap_{E^*} \, F_2^\bot$. Here $E^*$ is the dual vector bundle, and $F_1^\bot, F_2^\bot \subset E^*$ are the orthogonals to $F_1$ and $F_2$.

The main result of this paper is a further generalization of this equivalence, where now $V$ is replaced by a finite complex of locally free sheaves of arbitrary length (in non-positive degrees), on a scheme satisfying reasonable assumptions. (Schemes satisfying these conditions are called \emph{nice}, see \S \ref{ss:nice-schemes}; this class of schemes is much larger than the class to which \cite{MR} applies.) See Theorems~\ref{thm:Koszuldualitycomodules}, \ref{thm:equiv} and \ref{thm:equiv2} for three versions of this equivalence. This generalization uses ideas of Positselski (\cite{Po}), and does not rely on \cite{MR} except for some technical lemmas.

A version of this equivalence (when $F_2=E$) has been used by the second author in \cite{R2} to construct a Koszul duality for representations of semi-simple Lie algebras in positive characteristic. We have also used this construction in~\cite{MRIM} to provide a geometric realization of the Iwahori--Matsumoto involution of affine Hecke algebras. For other applications, see~\cite{UI} (in the setting of singularity categories); see also~\cite[Appendix F]{AG} for the relation with the definition of singular support for coherent sheaves.

One expects other applications in Algebraic Geometry akin to~\cite{UI}. For us the most important direction is to extend the present result to complexes of commutative group ind-schemes of arbitrary length. The case of length $2$ is Laumon's Fourier transform for $1$-motives~\cite{laumon}. Just as Laumon's transform is essential in the geometric Class Field Theory, we expect that the extension to arbitrary length will play the same role in an expected higher dimensional extension of the geometric Class Field Theory.

\subsection{Coherent sheaves on perfect sheaves}

If $\calE$ is a locally free $\calO_X$-module of finite rank, then it is well known that the category of quasi-coherent $\mathrm{S}_{\calO_X}(\calE^\vee)$-modules is equivalent to the category of quasi-coherent sheaves on the (unique) vector bundle whose sheaf of sections is $\calE$. By analogy, if $\calX$ is a finite complex of locally free $\calO_X$-modules of finite rank, then the derived category of quasi-coherent dg-modules over the (graded) symmetric algebra $\Sym_{\calO_X}(\calX^\vee)$ can be considered as the derived category of quasi-coherent sheaves on the complex $\calX$. So to any bounded complex $\calX$ of locally free $\calO_X$-modules on a nice scheme $X$ (in non-positive degrees) one can associate a derived category of quasi-coherent sheaves on $\calX$, and in this terminology linear Koszul duality is an equivalence between certain categories of quasi-coherent sheaves on $\calX$ and on $\calX^\vee[-1]$.

Assume moreover that $X$ is a $\mathbb{Q}$-scheme. Then if $\calX$ is a bounded complex of locally free sheaves, the complex $\Sym_{\calO_X}(\calX)$ is a direct factor of the tensor algebra $\mathrm{T}_{\calO_X}(\calX)$; it follows that if $\calX \to \calX'$ is a quasi-isomorphism of such complexes, the induced morphism $\Sym_{\calO_X}(\calX) \to \Sym_{\calO_X}(\calX')$ is a quasi-isomorphism\footnote{Note that this property does not hold over a field $k$ of characteristic $p$ (as pointed out to one of us by P.~Polo): for instance the symmetric algebra of the exact complex $(k \xrightarrow{\mathrm{id}_k} k)$ where the first copy is in even degree, is the de Rham complex of $\mathbb{A}^1_k$ (with possibly some shifts in the grading); in particular its cohomology is not $k$. This gives a simple example of the principle that the theory of commutative dg-algebras
is not appropriate outside of characteristic zero, which justifies the need for Lurie's formalism of Derived Algebraic Geometry.} (since such property is clear for the tensor algebra). If $X$ moreover admits an ample family of line bundles, then using \cite[Lemme 2.2.8.c]{SGA6} we deduce that for any perfect sheaf $\calF$ in $\calD^b \Coh(X)$ one can define the derived category of quasi-coherent sheaves on $\calF$ (which is well defined up to equivalence; i.e.~this category does not depend on the presentation of $\calF$ as a bounded complex of locally free sheaves), and then one can interpret linear Koszul duality in these terms. This justifies our title.

Note however that we will not use this point of view in the body of the paper. In particular, as our results make sense in arbitrary characteristic we will not restrict to $\mathbb{Q}$-schemes.

\subsection{Linear Koszul duality and Fourier--Sato transform}

In Sections~\ref{sec:morphisms} and \ref{sec:basechange} we study the behaviour of this linear Koszul duality under natural operations: we prove compatibility results with respect to morphisms of perfect sheaves and base change. These results are used in~\cite{MRIM} in the construction of the geometric realization of the Iwahori--Matsumoto involution of affine Hecke algebras.

These compatibility properties are inspired by the results of the second author in \cite[\S 2]{R2}, and are quite similar to some compatibility properties of the Fourier transform on constructible sheaves (see \cite[\S 3.7]{KS}). In fact, we make this similarity precise in \cite{MRFourier},
showing that our linear Koszul duality and a certain Fourier transform isomorphism in homology are related via the Chern character from equivariant $\Kth$-theory to (completed) equivariant Borel--Moore homology. This result explains the relation between the geometric realization of the Iwahori--Matsumoto involution in \cite{MRIM} and the geometric realization of the Iwahori--Matsumoto involution for Lusztig's \emph{graded} affine Hecke algebras constructed in \cite{EM}; see \cite{MRFourier} for details.

\subsection{Notation} 

If $X$ is a scheme, we denote by $\Sh(X)$ the category of all sheaves of $\calO_X$-modules. We denote by $\QCoh(X)$, respectively $\Coh(X)$ the category of quasi-coherent, respectively coherent, sheaves on $X$.


We will frequently work with $\Z^2$-graded sheaves $\calM$. The $(i,j)$ component of $\calM$ will be denoted $\calM^i_j$. Here ``$i$'' will be called the cohomological grading, and ``$j$'' will be called the internal grading. Ordinary sheaves will be considered as $\Z^2$-graded sheaves concentrated in bidegree $(0,0)$. As usual, if $\calM$ is a $\Z^2$-graded sheaf of $\calO_X$-modules, we denote by $\calM^\vee$ the $\Z^2$-graded $\calO_X$-module such that
\[
(\calM^\vee)^i_j \ := \ \sheafHom_{\calO_X}(\calM^{-i}_{-j},\calO_X).
\]
For $n \in \Z$, we also also denote by $\calM \langle n \rangle$ the $\Z^2$-graded module defined by
\[
(\calM \langle n \rangle)^i_j=\calM^i_{j-n}.
\]

As in \cite{MR} we will work with $\Gm$-equivariant sheaves of quasi-coherent $\calO_X$-dg-algebras over a scheme $X$ (for the trivial $\Gm$-action on $X$).\footnote{In a different terminology one would replace ``$\Gm$-equivariant dg-algebra'' by ``dgg-algebra'' (where dgg stands for differential graded graded). We chose this convention to emphasize the different roles played by the cohomological grading and the internal one.} If $\calA$ is such a dg-algebra, we denote by $\calC(\calA\Mod)$ the category of $\Gm$-equivariant quasi-coherent sheaves of $\calO_X$-dg-modules over $\calA$. (Beware that for simplicity we do not indicate ``$\Gm$'' or ``$\mathrm{gr}$'' in this notation, contrary to our conventions in \cite{MR,R2}.) We denote by $\calD(\calA\Mod)$ the associated derived category. On a few occasions we will also consider the category $\widetilde{\calC}(\calA\Mod)$ of \emph{all} sheaves of $\Gm$-equivariant $\calA$-dg-modules on $X$ (in the sense of \cite[\S 1.7]{R2}), and the associated derived category $\widetilde{\calD}(\calA\Mod)$.

If $X$ is a scheme and $\calF$ an $\calO_X$-module (considered as a bimodule where the left and right actions coincide), we denote by $\mathrm{S}_{\calO_X}(\calF)=\bigoplus_{i \in \Z_{\geq 0}} \mathrm{S}^i_{\calO_X}(\calF)$, respectively $\bigwedge_{\calO_X}(\calF) = \bigoplus_{i \in \Z} \bigwedge_{\calO_X}^i(\calF)$, the symmetric, respectively exterior, algebra of $\calF$, i.e.~the quotient of the tensor algebra of $\calF$ by the relations $f \otimes g - g \otimes f$, respectively 
$f \otimes f$,
for $f,g$ local sections of $\calF$. 
If $\calF$ is a complex of ($\Gm$-equivariant) $\calO_X$-modules, we denote by $\Sym_{\calO_X}(\calF)$ the graded-symmetric algebra of $\calF$, i.e.~the quotient of the tensor algebra of $\calF$ by the relations $f \otimes g - (-1)^{|f| \cdot |g|}g \otimes f$ for $f,g$ homogeneous local sections of $\calF$, together with $f \otimes f$ for local sections $f$ such that $|f|$ is odd.\footnote{Of course, this second set of equations is redundant if $2$ is invertible in $\Gamma(X,\calO_X)$.} (Here $|f|$ is the cohomological degree of the homogeneous element $f$.) This algebra is a sheaf of ($\Gm$-equivariant) dg-algebras in a natural way.

We will use the general convention that we denote similarly a functor between two categories and the induced functor between the opposite categories.


\subsection{acknowledgements}
This article is a sequel to \cite{MR}. This work was started while both authors were members of the Institute for Advanced Study in Princeton, during the Special Year on ``New Connections of Representation Theory to Algebraic Geometry and Physics.''

We thank L.~Positselski for his encouragements and for suggesting to look into \cite{Po} for solutions to our spectral sequence problems.

\section{Linear Koszul duality}
\label{sec:LKD}

\subsection{Nice schemes}
\label{ss:nice-schemes}

Most of our results will be proved under some technical assumptions on our base scheme. To simplify the statements we introduce the following terminology.

\begin{defin}
A scheme is \emph{nice} if it is separated, Noetherian, of finite Krull dimension, and if moreover it admits a dualizing complex in the sense of \cite[p.~258]{H}.
\end{defin}


If $X$ is a nice scheme, we will usually fix a dualizing object and denote it by $\Omega$. See \cite[\S 5.10]{H} for details about this assumption.

If $X$ is a nice scheme, it is proved in \cite{R2} that
for every graded-commutative, non-positively graded $\calO_X$-dg-algebra $\calA$, the category of sheaves of $\calA$-dg-modules has enough $K$-flat and $K$-injective objects in the sense of \cite{Sp}. If $\calA$ is moreover quasi-coherent, it is proved in \cite{BR} that the category of quasi-coherent sheaves of $\calA$-dg-modules also has enough $K$-flat and $K$-injective objects. If $\calA$ is $\Gm$-equivariant, then similar results hold for the categories $\widetilde{\calC}(\calA\Mod)$ and $\calC(\calA\Mod)$. (The case of $\widetilde{\calC}(\calA\Mod)$ is discussed in \cite[\S 1.7]{R2}; the case of $\calC(\calA\Mod)$ follows, using the techniques of \cite[\S 3]{BR}.) It follows that the usual functors (direct and inverse image, tensor product) admit derived functors. Moreover, these functors admit all the usual compatibility properties (see \cite[\S 3]{BR} for precise statements).

Recall finally that under our assumptions the category $\calD(\calA\Mod)$ depends on $\calA$ only up to quasi-isomorphism (see \cite[\S 3.6]{BR}).

\subsection{Koszul duality functors on the level of complexes}
\label{ss:definitions}

Let $(X,\calO_X)$ be a scheme. We consider a finite complex
\[
\calX:= \ \cdots 0 \to \calV^{-n} \to \calV^{-n+1} \to \cdots \to \calV^0 \to 0 \cdots
\]
where $n \geq 0$ and each $\calV^i$ is a locally free $\calO_X$-module of finite rank ($i \in \llbracket -n, 0 \rrbracket$). More precisely, we will consider $\calX$ as a complex of graded $\calO_X$-modules, where each $\calV^i$ is in internal degree $2$. Consider also the complex $\calY$ of graded $\calO_X$-modules which equals $\calX^{\vee}[-1]$ as a bigraded $\calO_X$-module, and where the differential is defined in such a way that
\[
d_{\calY}(y)(v) = (-1)^{|y|} y \bigl( d_{\calX}(v) \bigr)
\]
for $y$ a local section of $\calY$ and $v$ a local section of $\calX$. We define the (graded-commutative) $\Gm$-equivariant dg-algebras
\[
\calT:= \Sym_{\calO_X}(\calX) \quad \text{and} \quad \calS:= \Sym_{\calO_X}(\calY). 
\]
Our goal in this subsection is to construct functors
\[
\scra : \ \calC(\calT\Mod) \to \calC(\calS\Mod), \qquad \scrb : \ \calC(\calS\Mod) \to \calC(\calT\Mod).
\]

Let $\calM$ be in $\calC(\calT\Mod)$. We define $\scra(\calM)$ as follows. As a $\Z^2$-graded $\calO_X$-module, $\scra(\calM)$ equals $\calS \otimes_{\calO_X} \calM$. The $\calS$-action is induced by the left multiplication of $\calS$ on itself, and the differential is the sum of two terms, denoted $d_1$ and $d_2$. First, $d_1$ is the natural differential on the tensor product $\calS \otimes_{\calO_X} \calM$, defined by
\[
d_1(s \otimes m)=d_{\calS}(s) \otimes m + (-1)^{|s|} s \otimes d_{\calM}(m).
\]
Then, consider the natural morphism $i: \calO_X \to \sheafEnd_{\calO_X}(\calX) \cong \calX^\vee \otimes \calX$. The differential $d_2$ is the composition of the morphism
\[
\left\{
\begin{array}{ccc}
\calS \otimes_{\calO_X} \calM & \to & \calS \otimes_{\calO_X} \calM \\
s \otimes m & \mapsto & (-1)^{|s|} s \otimes m
\end{array}
\right.
\] 
followed by the morphism induced by $i$
\[
\calS \otimes_{\calO_X} \calM \to \calS \otimes_{\calO_X} \calX^\vee \otimes_{\calO_X} \calX \otimes_{\calO_X} \calM,
\]
and finally followed by the morphism
\[ 
\calS \otimes_{\calO_X} \calX^\vee \otimes_{\calO_X} \calX \otimes_{\calO_X} \calM \to \calS \otimes_{\calO_X} \calM
\]
induced by the right multiplication $\calS \otimes_{\calO_X} \calX^\vee \to \calS$ and the morphism $\calX \otimes_{\calO_X} \calM \to \calM$ induced by the $\calT$-action. Locally, one can choose a basis $\{x_{\alpha} \}$ of $\calX$ over $\calO_X$, and the dual basis $\{ x_{\alpha}^* \}$ of $\calX^{\vee}$; then $d_2$ can be described as
\[
d_2(s \otimes m) = (-1)^{|s|} \sum_{\alpha} s x_{\alpha}^* \otimes x_{\alpha} \cdot m.
\]
The following lemma is proved by a direct computation left to the reader.

\begin{lem}
\label{lem:definition-A}
These data provide $\scra(\calM)$ the structure of an $\calS$-dg-module.
\end{lem}

The functor $\scrb$ is constructed similarly. For $\calN$ in $\calC(\calS\Mod)$, $\scrb(\calN)$ is equal, as a $\Z^2$-graded $\calO_X$-module, to $\calT^\vee \otimes_{\calO_X} \calN$. The $\calT$-module structure is induced by the $\calT$-action on $\calT^\vee$ defined by $(t \cdot \phi)(t')=(-1)^{|t| \cdot |\phi|}\phi(t \cdot t')$.  And the differential is the sum of $d_1$, which is the differential of the tensor product $\calT^\vee \otimes_{\calO_X} \calN$, and $d_2$, the Koszul differential defined locally by
\[
d_2(\phi \otimes n) = (-1)^{|\phi|} \sum_{\alpha} \phi ( x_{\alpha} \cdot - ) \otimes x_{\alpha}^* \cdot n
\]
where, as above, $\{x_{\alpha} \}$ is a local basis of $\calX$ over $\calO_X$, and $\{ x_{\alpha}^* \}$ is the dual basis (of $\calX^{\vee}$).

The following lemma is proved by a direct computation left to the reader.

\begin{lem}
These data provide $\scrb(\calN)$ the structure of a $\calT$-dg-module.
\end{lem}

\begin{remark}
The assumption that $\calX$ is concentrated in non-positive degrees is not essential in~\S\S\ref{ss:definitions}--\ref{ss:kd-version3}. It will only simplify the definition of the duality functor in~\S\ref{ss:duality}.
\end{remark}

\subsection{Generalized Koszul complexes}

We consider the \emph{generalized Koszul complexes}
\[
\calK^{(1)}:= \scra(\calT^\vee), \qquad \calK^{(2)}:=\scrb(\calS).
\]
There exists a natural morphism of $\calS$-dg-modules $\calK^{(1)} \to \calO_X$ (projection to $\calO_X$ in bidegree $(0,0)$), and a natural morphism of $\calT$-dg-modules $\calO_X \to \calK^{(2)}$. The following proposition is similar to \cite[Lemma 2.6.1]{MR}.

\begin{prop} \label{prop:Koszulcomplex}
The natural morphism $\calK^{(1)} \to \calO_X$, respectively $\calO_X \to \calK^{(2)}$, is a quasi-isomorphism.
\end{prop}

\begin{proof}
The $\calO_X$-dg-modules $\calK^{(1)}$ and $\calK^{(2)}$ are isomorphic under the (graded) exchange of factors in the tensor product. Hence it is sufficient to treat the case of $\calK^{(1)}$. We prove the result by induction on $n$, the case $n=0$ being well known (see \cite[\S 2.3]{MR} and the references therein).

Assume $n>0$, and denote by $\calZ$ the complex
\[
\calZ:= \ \cdots 0 \to \calV_{-n+1} \to \cdots \to \calV_0 \to 0 \cdots
\]
Denote by $\calK^{(1)}_\calZ$ the Koszul complex for $\calZ$; by induction, its cohomology is concentrated in degree $0$, and equals $\calO_X$. To fix notation, assume $n$ is even. Then there is an isomorphism of graded sheaves 
\[
\calK^{(1)} \ \cong \ \bigoplus_{i,j,k} \,\bigwedge \hspace{-2.5pt} {}^i (\calV_{-n})^{\vee} \otimes_{\calO_X}  \bigl( \mathrm{S}^j(\calV_{-n}) \bigr)^{\vee} \otimes_{\calO_X} (\calK_\calZ^{(1)})^k,
\]
where the term $\bigwedge^i(\calV_{-n})^{\vee} \otimes_{\calO_X} \bigl( \mathrm{S}^j(\calV_{-n}) \bigr)^{\vee} \otimes_{\calO_X} (\calK_\calZ^{(1)})^k$ is in degree $k+nj+(n+1)i$. The differential on $\calK^{(1)}$ is the sum of four differentials: $d_1$ induced by the differential of $\calK_\calZ^{(1)}$, $d_2$ induced by the Koszul differential on the Koszul complex $\bigwedge(\calV_{-n})^{\vee} \otimes_{\calO_X} \bigl( \mathrm{S}(\calV_{-n}) \bigr)^{\vee}$, $d_3$ induced by $d_{\calX}^{-n} : \calV_{-n} \to \calV_{-n+1}$ acting in $\calT^\vee$, and $d_4$ induced by $d_{\calX}^{-n}$ acting in $\calS$. The effect of these differentials on the indices can be described as follows (where indices not indicated are fixed):
\begin{multline*}
d_1 : k \mapsto k+1, \quad 
d_2: \left\{ \begin{array}{ccc} i & \mapsto & i+1 \\ j & \mapsto & j-1 \end{array} \right. , \\
d_3: \left\{ \begin{array}{ccc} j & \mapsto & j+1 \\ k & \mapsto & k-n+1 \end{array} \right. , \quad 
d_4: \left\{ \begin{array}{ccc} i & \mapsto & i+1 \\ k & \mapsto & k-n \end{array} \right. .
\end{multline*}
Hence $\calK^{(1)}$ is the total complex of the double complex whose $(p,q)$-term is
\[
\calA^{p,q}:= \bigoplus_{\genfrac{}{}{0pt}{}{p=i+j,}{q=k+ni+(n-1)j}} \bigwedge \hspace{-2.5pt} {}^i (\calV_{-n})^{\vee} \otimes_{\calO_X} \bigl( \mathrm{S}^j(\calV_{-n}) \bigr)^{\vee} \otimes_{\calO_X} (\calK_\calZ^{(1)})^k,
\]
and whose differentials are $d'=d_3 + d_4$, $d''=d_1+d_2$. By definition we have $\calA^{p,q}=0$ if $q<0$. Hence there is a converging spectral sequence 
\[
E_1^{p,q} = \calH^q(\calA^{p,*}, d'') \ \Rightarrow \ \calH^{p+q}(\calK^{(1)})
\]
(see \cite[Proposition 2.2.1$\rmii$]{MR}). It follows that that it suffices to prove that the complex $\calK^{(1)}$, endowed with the differential $d_1+d_2$ (i.e.~the ordinary tensor product of the Koszul complex of $(\calV_{-n})^\vee$ and $\calK_\calZ^{(1)})$, has cohomology $\calO_X$. 

The Koszul complex $\bigwedge(\calV_{-n})^{\vee} \otimes_{\calO_X} \bigl( \mathrm{S}(\calV_{-n}) \bigr)^{\vee}$ is a $K$-flat object of $\calC(\calO_X\Mod)$ in the sense of \cite{Sp} (since each of its internal degree components is a bounded complex of flat $\calO_X$-modules). It follows that the morphism
\[
\Bigl( \bigwedge(\calV_{-n})^{\vee} \otimes_{\calO_X} \bigl( \mathrm{S}(\calV_{-n}) \bigr)^{\vee} \Bigr) \otimes_{\calO_X} \calK_\calZ^{(1)} \to \Bigl( \bigwedge(\calV_{-n})^{\vee} \otimes_{\calO_X} \bigl( \mathrm{S}(\calV_{-n}) \bigr)^{\vee} \Bigr) \otimes_{\calO_X} \calO_X
\]
is a quasi-isomorphism. The claim follows, since the Koszul complex of $(\calV_{-n})^\vee$ has cohomology $\calO_X$ (see again \cite[\S 2.3]{MR}).
\end{proof}

\subsection{Covariant linear Koszul duality}
\label{ss:KD1}

We denote by $\calC(\calT\Mod_-)$ the full subcategory of $\calC(\calT\Mod)$ whose objects have their internal degree which is bounded above (uniformly in the cohomological degree), and by $\calD(\calT\Mod_-)$ the corresponding derived category. We define similarly $\calC(\calS\Mod_-)$ and $\calD(\calS\Mod_-)$. Note that $\calT^\vee$ and $\calS$ are both concentrated in non-positive internal degrees. In particular, they are objects of the categories just defined. Note also that the natural functors 
\[
\calD(\calT\Mod_-) \to \calD(\calT\Mod) \quad \text{and} \quad \calD(\calS\Mod_-) \to \calD(\calS\Mod)
\]
are fully faithful, with essential images the subcategories of dg-modules whose cohomology is bounded above for the internal degree (uniformly in the cohomological degree).

The first (covariant) version of our linear Koszul duality is the following theorem.

\begin{thm}
\label{thm:Koszuldualitycomodules}
\begin{enumerate}
\item The functors
\[
\scra: \calC(\calT\Mod_-) \to \calC(\calS\Mod_-), \qquad \scrb : \calC(\calS\Mod_-) \to \calC(\calT\Mod_-)
\]
are exact, hence induce functors
\[
\overline{\scra}: \calD(\calT\Mod_-) \to \calD(\calS\Mod_-), \qquad \overline{\scrb} : \calD(\calS\Mod_-) \to \calD(\calT\Mod_-).
\]
\item The functors $\overline{\scra}$ and $\overline{\scrb}$
are equivalences of triangulated categories, quasi-inverse to each other.
\end{enumerate}
\end{thm}

\begin{proof}
The idea of the proof is taken from \cite[proof of Theorem~A.1.2]{Po}.

$\rmi$ We prove that $\scra$ sends acyclic $\calT$-dg-modules in $\calC(\calT\Mod_-)$ to acyclic $\calS$-dg-modules; the proof for $\scrb$ is similar. Let $\calM$ be an acyclic object of $\calC(\calT\Mod_-)$. To fix notation, assume that $\calM_n=0$ for $n>0$. Then for any $m \in \Z$ we have
\begin{equation*}
\scra(\calM)_m=\bigoplus_{\genfrac{}{}{0pt}{}{i \leq 0, j \leq 0,}{m=i+j}} \calS_i \otimes_{\calO_X} \calM_j.
\end{equation*}
In particular, $\scra(\calM)_m=0$ unless $m \leq 0$.

For $m \leq 0$, the complex
$\scra(\calM)_m$ can be obtained from the finite collection of complexes $\calS_i \otimes_{\calO_X} \calM_j$ with $i,j \leq 0$ and $i+j=m$ (with their natural differential as a tensor product of complexes) by taking shifts and cones. More precisely, this process works as follows, e.g.~when $m$ is even: start with $\calS_m \otimes \calM_0$; Lemma~\ref{lem:definition-A} implies that the Koszul differential (denoted $d_2$ in~\S\ref{ss:definitions}) defines a morphism of complexes $\calS_{m+2} \otimes \calM_{-2}[-1] \to \calS_m \otimes \calM_0$; take its cone $\calL$; then similarly the Koszul differential defines a morphism of complexes $\calS_{m+4} \otimes \calM_{-4}[-1] \to \calL$; take its cone; and continue until $\calS_0 \otimes \calM_m$ is reached.


As $\calM_j$ is acyclic for any $j$, and as the tensor product of an acyclic complex with a bounded above complex of flat modules is acyclic, the complexes $\calS_i \otimes_{\calO_X} \calM_j$ are acyclic. Then,
as the cone of a morphism between acyclic complexes is acyclic, we deduce from the above observation that $\scra(\calM)_m$ is acyclic for any $m$, and finally that $\scra(\calM)$ is acyclic.

$\rmii$ The functors $\scra$ and $\scrb$ are clearly adjoint; in particular we have adjunction morphisms $\scra \circ \scrb \to \Id$ and $\Id \to \scrb \circ \scra$, hence similar morphisms for the induced functors $\overline{\scra}$ and $\overline{\scrb}$. We show that the morphism $\overline{\scra} \circ \overline{\scrb} \to \Id$ is an isomorphism; the proof for the morphism $\Id \to \overline{\scrb} \circ \overline{\scra}$ is similar. Let $\calN$ be an object of $\calC(\calS\Mod_-)$. By a construction similar to that of $\rmi$, the homogeneous internal degree components of the cone of the morphism $\scra \circ \scrb (\calN) \to \calN$ can be obtained from the negative homogeneous internal degree components of $\calK^{(1)}$ by tensoring with some homogeneous internal degree components of $\calM$ and taking shifts and cones a finite number of times. Hence, using Proposition~\ref{prop:Koszulcomplex}, it suffices to observe that the tensor product of an acyclic, bounded above complex of flat $\calO_X$-modules with any complex of $\calO_X$-modules is acyclic (see e.g.~\cite[Proposition 5.7]{Sp}).
\end{proof}

\begin{remark}
\begin{enumerate}
\item
So far, we have not used the condition that our dg-modules 
are quasi-coherent over $\calO_X$, or that $(X,\calO_X)$ is a scheme. In fact, Theorem~\ref{thm:Koszuldualitycomodules} is true for any commutative ringed space $(X,\calO_X)$ and the whole categories of $\calO_X$-dg-modules over $\calT$ and $\calS$ (with the prescribed condition on the grading). Our assumption will be used in \S \ref{ss:KD-v2} below.
\item
One can easily check using Proposition~\ref{prop:Koszulcomplex} that the composition of $\overline{\scra}$ with the forgetful functor $\calD(\calS\Mod_-) \to \calD(\calO_X\Mod)$ is isomorphic to $R\sheafHom_{\calT}(\calO_X,-)$, i.e.~to the right derived functor of the functor $\calM \mapsto \sheafHom_{\calT}(\calO_X, \calM)$. Similarly the composition of $\overline{\scrb}$ with the forgetful functor can be identified with $\calO_X \lotimes_{\calS} (-)$.
\end{enumerate}
\end{remark}

\subsection{Reminder on Grothendieck--Serre duality}
\label{ss:duality}

From now on we assume that $X$ is a nice scheme. Let $\Omega \in \calD^b \Coh(X)$ be a dualizing object for $X$. We choose a bounded below complex $\calI_\Omega$ of injective quasi-coherent sheaves on $X$ whose image in $\calD^b \Coh(X)$ is $\Omega$. Recall that the components of $\calI_\Omega$ are also injective in the category $\Sh(X)$. (This follows from \cite[Theorem II.7.18]{H}.) 

The functor
\[
\sheafHom_{\calO_X}(-,\calI_\Omega) : \calC \Sh(X) \to \calC \Sh(X)^\op
\]
is exact, hence induces a functor
\[
\sheafHom_{\calO_X}(-,\calI_\Omega) : \calD \Sh(X) \to \calD \Sh(X)^\op.
\]
It is well known that under our assumptions the natural functor
\begin{equation}
\label{eqn:qcoh-sh}
\calD \QCoh(X) \to \calD \Sh(X)
\end{equation}
is fully faithful (see e.g.~\cite[Proposition 3.1.3]{BR} and the references therein), as well as the natural functor
\begin{equation}
\label{eqn:coh-qcoh}
\calD^b \Coh(X) \to \calD \QCoh(X).
\end{equation}
Hence we can consider the category $\calD^b \Coh(X)$ as a subcategory of  $\calD \Sh(X)$, hence obtain a functor
\[
\sheafHom_{\calO_X}(-,\calI_\Omega) : \calD^b \Coh(X) \to \calD \Sh(X)^\op.
\]
It is explained in \cite[p.~257]{H} that this functor factors through a functor
\[
\mathrm{D}_\Omega : \calD^b \Coh(X) \to \calD^b \Coh(X)^\op.
\]
Using again the fully faithfulness of \eqref{eqn:qcoh-sh} and \eqref{eqn:coh-qcoh}, it is easy to construct a morphism $\varepsilon_\Omega : \Id \to  \mathrm{D}_\Omega \circ \mathrm{D}_\Omega$ of endofunctors of the category $\calD^b \Coh(X)$. The fact that $\Omega$ is a dualizing complex implies that $\varepsilon_\Omega$ is an isomorphism of functors (see \cite[Proposition V.2.1]{H}). In particular, $\mathrm{D}_\Omega$ is an equivalence of categories.

Now we let $\calA$ be a $\Gm$-equivariant, non-positively (cohomologically) graded, graded-commu\-tative sheaf of $\calO_X$-dg-algebras on $X$. We denote by $\calD^\bc (\calA\Mod)$ the subcategory of $\calD(\calA\Mod)$ with objects the dg-modules $\calM$ such that for any $j \in \Z$ the object $\calM_j$ of $\calD\QCoh(X)$ has bounded and coherent cohomology. Our goal now is to explain the construction of an equivalence
\[
\mathrm{D}_\Omega^\calA : \calD^\bc (\calA\Mod) \xrightarrow{\sim} \calD^\bc (\calA\Mod)^\op
\]
which is compatible with $\mathrm{D}_\Omega$ in the natural sense. 
First, there exists a natural functor
\[
{}^0 \widetilde{\mathrm{D}}_\Omega^\calA : \widetilde{\calC}(\calA\Mod) \to \widetilde{\calC}(\calA\Mod)^\op
\]
which sends a dg-module $\calM$ to the dg-module whose underlying $\Gm$-equivariant $\calO_X$-dg-module is $\sheafHom_{\calO_X}(\calM,\calI_\Omega)$, with $\calA$-action defined by
\[
(a \cdot \phi)(m) = (-1)^{|a| \cdot |\phi|} \phi(a \cdot m)
\]
for $a$ a local section of $\calA$, $\phi$ a local section of $\sheafHom_{\calO_X}(\calM,\calI_\Omega)$, and $m$ a local section of $\calM$. This functor is exact, hence induces a functor
\[
\widetilde{\mathrm{D}}_\Omega^\calA : \widetilde{\calD}(\calA\Mod) \to \widetilde{\calD}(\calA\Mod)^\op.
\]
Moreover, it is easy to construct a morphism $\widetilde{\varepsilon}_\Omega^\calA : \Id \to \widetilde{\mathrm{D}}_\Omega^\calA \circ \widetilde{\mathrm{D}}_\Omega^\calA$ of endofunctors of $\widetilde{\calD}(\calA\Mod)$. Now by \cite[Proposition 3.3.2]{BR} the natural functor $\calD(\calA\Mod) \to \widetilde{\calD}(\calA\Mod)$ is fully faithful, hence so is also the natural functor $\calD^\bc(\calA\Mod) \to \widetilde{\calD}(\calA\Mod)$. Moreover, it is easy to prove using the case of $\calO_X$ considered above that $\widetilde{\mathrm{D}}_\Omega^\calA$ factors through a functor
\[
\mathrm{D}_\Omega^\calA : \calD^\bc (\calA\Mod) \to \calD^\bc (\calA\Mod)^\op.
\]
Using fully faithfulness again the morphism $\widetilde{\varepsilon}_\Omega^\calA$ induces a morphism $\varepsilon_\Omega^\calA : \Id \to \mathrm{D}_\Omega^\calA \circ \mathrm{D}_\Omega^\calA$ of endofunctors of $\calD^\bc(\calA\Mod)$. Finally, using again the case of $\calO_X$ one can check that $\varepsilon_\Omega^\calA$ is an isomorphism of functors, proving in particular that $\mathrm{D}_\Omega^\calA$ is an equivalence.

\subsection{Contravariant linear Koszul duality}
\label{ss:KD-v2}

As in \S\ref{ss:duality} we assume that $X$ is nice. Using the same conventions as in \S\ref{ss:duality}, we denote by $\calD^\bc(\calT\Mod_-)$ the subcategory of the category $\calD(\calT\Mod_-)$ whose objects are the dg-modules $\calM$ such that, for any $j \in \Z$, $\calH^\bullet(\calM_j)$ is bounded and coherent. We define similarly the categories $\calD^\bc(\calS\Mod_-)$ (replacing $\calT$ by $\calS$) and $\calD^\bc(\calT\Mod_+)$ (replacing ``above'' by ``below'').

Clearly, the functors $\overline{\scra}$ and $\overline{\scrb}$ restrict to equivalences
\[
\overline{\scra}^\bc: \calD^\bc(\calT\Mod_-) \xrightarrow{\sim} \calD^\bc(\calS\Mod_-), \qquad \overline{\scrb}^\bc : \calD^\bc(\calS\Mod_-) \xrightarrow{\sim} \calD^\bc(\calT\Mod_-)
\]
(see e.g.~the proof of Theorem~\ref{thm:Koszuldualitycomodules}). On the other hand, in \S\ref{ss:duality} we have defined an equivalence $\mathrm{D}_\Omega^\calT$, which induces an equivalence of categories
\[
\mathbf{D}_\Omega^\calT : \calD^\bc(\calT\Mod_+) \ \xrightarrow{\sim} \ \calD^\bc(\calT\Mod_-)^\op.
\]
Composing these equivalences we obtain the following result, which is the second (contravariant) version of our linear Koszul duality.

\begin{thm}
\label{thm:equiv}
Let $X$ be a nice scheme with dualizing complex $\Omega$. Then the composition $\overline{\scra}^{\bc} \circ \mathbf{D}_{\Omega}^\calT$ gives an equivalence of triangulated categories
\[
K_{\Omega} :  \calD^{\bc}(\calT\Mod_+) \ \xrightarrow{\sim} \ \calD^{\bc}(\calS\Mod_-)^\op,
\]
which satisfies $K_{\Omega}(\calM[n] \langle m \rangle) = K_{\Omega}(\calM)[-n]\langle -m \rangle$.
\end{thm}

\begin{remark}
\label{rk:thm-equiv}
\begin{enumerate}
\item
The equivalence $K_{\Omega}$ only depends, up to isomorphism, on the dualizing complex $\Omega \in \calD^b \Coh(X)$, and not on the injective resolution $\calI_{\Omega}$. This justifies the notation.
\item
\label{it:description-K}
One can describe the equivalence $K_{\Omega}$ very explicitly. Namely, if $\calM$ is an object of the category $\calD^\bc(\calT\Mod_+)$, the image under $K_{\Omega}$ of $\calM$ is, as an $\calS$-dg-module, the image in the derived category of the complex
\[
\calS \otimes_{\calO_X} \sheafHom_{\calO_X}(\calM,\calI_{\Omega}),
\]
where the differential is the sum of a Koszul-type differential and of the usual differential on the tensor product $\calS \otimes \sheafHom(\calM,\calI_{\Omega})$. (Note that this dg-module might not be quasi-coherent.)
\item 
The covariant Koszul duality $\overline{\scra}$ may seem more appealing than the contravariant duality $K_{\Omega}$. However, the category $\calD(\calT\Mod_-)$ is not very interesting since it does not contain the free module $\calT$ in general. In particular, the equivalence $\overline{\scra}$ will not yield an equivalence for locally finitely generated dg-modules, contrary to $K_{\Omega}$ (see Proposition~\ref{prop:restriction-kappa} below). This equivalence will be interesting, however, if $\calX$ is concentrated in odd cohomological degrees. In this case, the equivalence obtained is essentially that of \cite[Theorem 8.4]{GKM} (see also \cite[Theorem 2.1.1]{R2}).
\end{enumerate}
\end{remark}

\subsection{Regraded contravariant linear Koszul duality}
\label{ss:kd-version3}

Consider the $\Gm$-equivariant dg-algebra
\[
\calR:=\Sym(\calY[2]).
\]
There is a ``regrading'' equivalence of categories
\[
\xi : \calC(\calS\Mod) \xrightarrow{\sim} \calC(\calR\Mod),
\]
which sends the $\calS$-dg-module $\calN$ to the $\calR$-dg-module with $(i,j)$-component $\xi(\calM)^i_j:=\calM^{i-j}_j$. (If one forgets the gradings, the dg-algebras $\calR$ and $\calS$ coincide, as well as $\calM$ and $\xi(\calM)$. Then the $\calR$-action and the differential on $\xi(\calM)$ are the same as the $\calS$-action and the differential on $\calM$.) Composing the equivalence of Theorem~\ref{thm:equiv} with $\xi$ we obtain the third version of our linear Koszul duality, which is the one we will use.

\begin{thm}
\label{thm:equiv2}
Let $X$ be a nice scheme with dualizing complex $\Omega$. Then the composition $\xi \circ K_{\Omega}$ gives an equivalence of triangulated categories
\[
\kappa_{\Omega} :  \calD^{\bc}(\calT\Mod_+) \ \xrightarrow{\sim} \ \calD^{\bc}(\calR\Mod_-)^\op,
\]
which satisfies $\kappa_{\Omega}(\calM[n] \langle m \rangle) = \kappa_{\Omega}(\calM)[-n+m]\langle -m \rangle$.
\end{thm}

\subsection{Finiteness conditions}
\label{ss:finiteness}

From now on, for simplicity we assume that $n \leq 1$.

We will now consider subcategories of locally finitely generated $\calT$- and $\calR$-dg-modules. More precisely, 
we define $\calD^{\fg}(\calT\Mod)$, respectively $\calD^{\fg}(\calT\Mod_+)$, as the subcategory of $\calD(\calT\Mod)$, respectively $\calD(\calT\Mod_+)$, whose objects are the dg-modules whose cohomology is locally finitely generated over the cohomology $\calH^{\bullet}(\calT)$. We use similar notation for $\calR$-dg-modules. It is easily checked that the natural functors induce equivalences
\begin{equation}
\label{eqn:equiv-fg}
\calD^{\fg}(\calT\Mod_+) \xrightarrow{\sim} \calD^{\fg}(\calT\Mod), \qquad \calD^{\fg}(\calR\Mod_-) \xrightarrow{\sim} \calD^{\fg}(\calR\Mod).
\end{equation}

%
%
%


\begin{prop} \label{prop:restriction-kappa}
Assume $X$ is a nice scheme with dualizing complex $\Omega$, and that $n \leq 1$. Then $\kappa_{\Omega}$ restricts to an equivalence of triangulated categories
\[
\calD^{\fg}(\calT\Mod_+) \ \xrightarrow{\sim} \ \calD^{\fg}(\calR\Mod_-)^\op. 
\]
\end{prop}

\begin{proof}
\emph{First case:} $n=0$. In this case, $\calT=\mathrm{S}(\calV_0)$, with generators in bidegree $(0,2)$, and $\calR=\bigwedge(\calV_0)^{\vee}$, with generators in bidegree $(-1,-2)$. We need to check that $\kappa_\Omega$ sends $\calD^{\fg}(\calT\Mod_+)$ into $\calD^{\fg}(\calR\Mod_-)^\op$, and that $\kappa_\Omega^{-1}$ sends $\calD^{\fg}(\calR\Mod_-)^\op$ into $\calD^{\fg}(\calT\Mod_+)$.

First, consider the case of $\kappa_\Omega^{-1}$. Since $\calR$ is concentrated in non-positive cohomological degrees, the usual truncation functors for complexes make sense for $\calR$-dg-modules. In particular, since any object in $\calD^{\fg}(\calR\Mod_-)$ has only finitely many non-zero cohomology objects, this implies that $\calD^{\fg}(\calR\Mod_-)$ is generated, as a triangulated category, by objects whose cohomology is concentrated in only one degree. This also implies that such an object is isomorphic to a dg-module which is non-zero in only one degree. Such an object has a trivial $\calR$-action, and it is easily checked that its image under $\kappa_\Omega^{-1}$ is of the form $\calT \otimes_{\calO_X} \calF$ for $\calF$ in $\calD^b \Coh^{\Gm}(X)$; in particular it belongs to $\calD^{\fg}(\calT\Mod_+)$.

Now we consider the case of $\kappa_{\Omega}$. By Remark~\ref{rk:thm-equiv}\ref{it:description-K} and the properties of dualizing complexes (recalled in~\S\ref{ss:duality}), it is enough to prove that if $\calM$ belongs to $\calD^{\fg}(\calT\Mod_+)$ then the complex $\calS^\vee \otimes_{\calO_X} \calM = (\calS^\vee \otimes_{\calO_X} \calT) \otimes_{\calT} \calM$ (where $\calS^\vee \otimes_{\calO_X} \calT$ is endowed with the appropriate Koszul differential making it quasi-isomorphic to $\calO_X$) 
belongs to $\calD^b \Coh^{\Gm}(X)$.
This claim is local, hence we can assume that $X$ is affine and $\calV_0$ is free.

In this setting, our claim follows from the following fact: if $A$ is a Noetherian ring and $r \in \Z_{\geq 0}$, for any finitely generated $A[X_1, \cdots, X_r]$-module $M$, the $A$-modules $\mathsf{H}^i(A \lotimes_{A[X_1, \cdots, X_r]} M)$ are finitely generated, and $0$ except for finitely many $i$'s. (Here, $A$ has the ``trivial'' module structure, where each $X_j$ acts by $0$.) In fact, finite generation follows from the fact that $M$ admits a (possibly unbounded) resolution by finitely generated flat $A[X_1, \cdots, X_r]$-modules, and vanishing except for finitely many $i$'s follows from the fact that $A$ admits a finite flat resolution over $A[X_1, \cdots, X_r]$ (for instance the Koszul resolution).

\emph{Second case:} $n=1$. In this case $\calT=\Sym(\calV_{-1} \to \calV_0)$, with $\calV_{-1}$ in bidegree $(-1,2)$ and $\calV_0$ in bidegree $(0,2)$, while $\calR=\Sym \bigl( (\calV_0)^{\vee} \to (\calV_{-1})^{\vee} \bigr)$, with $(\calV_0)^{\vee}$ in bidegree $(-1,-2)$ and $(\calV_{-1})^{\vee}$ in bidegree $(0,-2)$. As in \S \ref{ss:duality}, we denote by $\calI_{\Omega}$ a bounded below complex of quasi-coherent injective $\calO_X$-modules whose image in the derived category is $\Omega$. 

As in the first case we need to show that $\kappa_\Omega$ sends $\calD^{\fg}(\calT\Mod_+)$ into $\calD^{\fg}(\calR\Mod_-)^\op$. Using truncation functors again, for this it suffices to prove that for any locally finitely generated $\Gm$-equivariant $\calT$-dg-module $\calM$, $\kappa_\Omega(\calM)$ belongs to $\calD^{\fg}(\calR\Mod_-)^\op$.
And in turn, for this it suffices to prove that the
cohomology of $\kappa_\Omega(\calM)$ is locally finitely generated over $\mathrm{S} \bigl( (\calV_{-1})^{\vee} \bigr)$. In fact it will be equivalent and easier to work with the equivalence $K_{\Omega}$ of Theorem~\ref{thm:equiv} and the dg-algebra $\calS$ (whose generators are in bidegrees $(1,-2)$ and $(2,-2)$).

Let $\widetilde{K}_{\Omega}$ be the equivalence corresponding to the complex of $\calO_X$-modules $\widetilde{\calX}$ concentrated in degree $0$, with only non-zero component $\calV_0$. Denote by $\widetilde{\calT}$, $\widetilde{\calS}$ the dg-algebras defined similarly to $\calT$ and $\calS$, but for the complex $\widetilde{\calX}$ instead of $\calX$. With this notation, $K_{\Omega}(\calM)$ is the image in the derived category of the $\calS$-dg-module $\calS \otimes_{\calO_X} \sheafHom_{\calO_X}(\calM,\calI_{\Omega})$ (with a certain differential), and $\widetilde{K}_{\Omega}(\calM)$ is the image in the derived category of the $\widetilde{\calS}$-dg-module $\widetilde{\calS} \otimes_{\calO_X} \sheafHom_{\calO_X}(\calM,\calI_{\Omega})$.

The locally finitely generated $\calT$-dg-module $\calM$ is also locally finitely generated as a $\widetilde{\calT}$-dg-module. By the first case, we deduce that $\widetilde{K}_{\Omega}(\calM)$ has locally finitely generated cohomology over $\widetilde{\calS}$; in other words, this cohomology is bounded and coherent over $\calO_X$. Now we have an isomorphism of $\mathrm{S} \bigl( (\calV_{-1})^{\vee} \bigr)$-dg-modules
\begin{equation}
\label{eq:decomposition}
K_{\Omega}(\calM) \ \cong \ \bigoplus_{i,j} \, \mathrm{S}^i \bigl( (\calV_{-1})^{\vee} \bigr) \otimes_{\calO_X} \bigl( \widetilde{K}_{\Omega}(\calM) \bigr)^j,
\end{equation}
where the term $\mathrm{S}^i \bigl( (\calV_{-1})^{\vee} \bigr) \otimes_{\calO_X} \bigl( \widetilde{K}_{\Omega}(\calM) \bigr)^j$ is in cohomological degree $j+2i$. The differential on $K_{\Omega}(\calM)$ is the sum of three terms: the differential $d_1$ induced by $d_{\calS}$, the Koszul differential $d_2$ induced by that of the Koszul complex $(\bigwedge \calV_{-1})^{\vee} \otimes \mathrm{S} \bigl( (\calV_{-1})^{\vee} \bigr)$, and finally the differential $d_3$ induced by that of $\widetilde{K}_{\Omega}(\calM)$. The effect of these differentials on the degrees of the decomposition \eqref{eq:decomposition} are the following:
\[
d_1 : \left\{
\begin{array}{ccc}
i & \mapsto & i+1 \\
j & \mapsto & j-1
\end{array} \right. ,
\quad d_2 : \left\{
\begin{array}{ccc}
i & \mapsto & i+1 \\
j & \mapsto & j-1
\end{array} \right. ,
\quad d_3 : j \mapsto j+1.
\]
Hence $K_{\Omega}(\calM)$ is the total complex of the double complex with $(p,q)$-term
\[
\calB^{p,q}:= \bigoplus_{\genfrac{}{}{0pt}{}{p=i,}{q=j+i}} \mathrm{S}^i \bigl( (\calV_{-1})^{\vee} \bigr) \otimes_{\calO_X} \bigl( \widetilde{K}_{\Omega}(\calM) \bigr)^j
\]
and differentials $d':=d_1+d_2$, $d'':=d_3$. Now $\calM$ is locally finitely generated over $\calT$; in particular it is bounded above for the cohomological grading. Hence $\widetilde{K}_{\Omega}(\calM)$ is bounded below for the cohomological grading. It follows that $\calB^{p,q}=0$ for $q \ll 0$. Hence by \cite[Proposition 2.2.1]{MR}, there exists a converging spectral sequence
\[
E_1^{p,q}=\calH^q(\calB^{p,*},d'') \ \Rightarrow \ \calH^{p+q}(K_{\Omega}(\calM)).
\]
As $\widetilde{K}_{\Omega}(\calM)$ has bounded, coherent cohomology over $\calO_X$, $E_1^{\bullet,\bullet}$ is a locally finitely generated $\mathrm{S}((\calV_{-1})^{\vee})$-module. Moreover, in the $(p,q)$-plane, it is concentrated on an ascending diagonal strip. It follows that the spectral sequence is stationary after finitely many steps. We deduce that $\calH^{\bullet}(K_{\Omega}(\calM))$ is locally finitely generated over $\mathrm{S}((\calV_{-1})^{\vee})$.


By symmetry the equivalence $(\kappa_{\Omega})^{-1}$ also sends dg-modules with locally finitely generated cohomology to dg-modules with the same property, which finishes the proof.
\end{proof}

\subsection{Intersection of vector subbundles}
\label{ss:lkd}

We can now extend the main result of \cite{MR} to nice schemes. We let $X$ be such a scheme, and denote by $\Omega$ a dualizing complex. 

Let $E$ be a vector bundle over $X$, let $F_1,F_2 \subset E$ be two vector subbundles, and let $\calF_1,\calF_2,\calE$ be the sheaves of sections of $F_1,F_2,E$. Let $E^*$ be the dual vector bundle, and $F_1^{\bot}, F_2^{\bot} \subset E^*$ be the orthogonals to $F_1$ and $F_2$. We will apply the constructions above to the complex
\begin{equation}
\label{eqn:definition-X-lkd}
\calX:=(0 \to \calF_1^{\bot} \to \calF_2^{\vee} \to 0)
\end{equation}
where $\calF_1^\bot$ is in degree $-1$, $\calF_2^\vee$ is in degree $0$, and the differential is the composition $\calF_1^\bot \hookrightarrow \calE^\vee \twoheadrightarrow \calF_2^\vee$.
We set
\[
\calD^{\mathrm{c}}_{\Gm}(F_1 \, \rcap_E \, F_2) \ := \ \calD^{\fg}(\calT\Mod), \qquad
\calD^{\mathrm{c}}_{\Gm}(F_1^{\bot} \, \rcap_{E^*} \, F_2^{\bot}) \ := \ \calD^{\fg}(\calR\Mod).
\]

To justify this notation, recall the notion of dg-scheme, first defined in \cite{CK} and later studied in \cite{R2}, \cite{MR}, \cite{BR}. (Here we follow the conventions of \cite{MR}.) It is explained in \cite[Lemma 4.1.1]{MR} that the derived category of coherent dg-sheaves on the dg-scheme $F_1 \, \rcap_E \, F_2$, respectively $F_1^{\bot} \, \rcap_{E^*} \, F_2^{\bot}$, is equivalent to the subcategory of the derived category of quasi-coherent dg-modules over the sheaf of $\calO_X$-dg-algebras $\calT$, respectively $\calR$, whose objects have their cohomology locally finitely generated. (Note that here we consider \emph{ordinary} dg-modules, and not $\Gm$-equivariant ones.) Hence the category $\calD^{\fg}(\calT\Mod)$, respectively $\calD^{\fg}(\calR\Mod)$, is a ``graded version'' of this category.

\begin{remark}
\label{rk:symmetry-definition}
Our definition of the category $\calD^{\mathrm{c}}_{\Gm}(F_1 \, \rcap_E \, F_2)$ is not symmetric in $F_1$ and $F_2$. However, it follows from an obvious $\Gm$-equivariant analogue of \cite[Proposition 1.3.2]{MR} that the triangulated categories $\calD^{\mathrm{c}}_{\Gm}(F_1 \, \rcap_E \, F_2)$ and $\calD^{\mathrm{c}}_{\Gm}(F_2 \, \rcap_E \, F_1)$ are equivalent.
\end{remark}

With this notation, using~\eqref{eqn:equiv-fg}, Proposition~\ref{prop:restriction-kappa} can rephrased in the following terms.

\begin{thm}
\label{thm:lkd}
Assume that $X$ is a nice scheme with dualizing complex $\Omega$. Then $\kappa_{\Omega}$ induces an equivalence of triangulated categories
\[
\kappa_{\Omega} : \calD^{\mathrm{c}}_{\Gm}(F_1 \, \rcap_E \, F_2) \ \xrightarrow{\sim} \ \calD^{\mathrm{c}}_{\Gm}(F_1^{\bot} \, \rcap_{E^*} \, F_2^{\bot})^\op,
\]
which satisfies $\kappa_{\Omega}(\calM[n] \langle m \rangle) = \kappa_{\Omega}(\calM)[-n+m]\langle -m \rangle$.
\end{thm}

\begin{remark}
One can easily check that, if the assumptions of \cite[Theorem 4.2.1]{MR} are satisfied, then $\calO_X$ is a dualizing complex, and the equivalence of \emph{loc.}~\emph{cit.}~is isomorphic to the equivalence $\kappa_{\calO_X}$ of Theorem~\ref{thm:lkd}.
\end{remark}

\section{Linear Koszul duality and morphisms of perfect complexes} \label{sec:morphisms}

\subsection{Statement} \label{ss:kd-morphisms-complexes}

Let us come back to the setting of \S \ref{ss:kd-version3}. More precisely, we consider a nice scheme $X$ with dualizing complex $\Omega$, and two complexes $\calX$ and $\calX'$ of locally free sheaves as in \S\ref{ss:definitions}. 
We denote by $\calR,\calS,\calT$, respectively $\calR',\calS',\calT'$, the dg-algebras constructed from $\calX$, respectively $\calX'$. We also denote by
\[
\kappa_{\Omega} :  \calD^{\bc}(\calT\Mod_+) \ \xrightarrow{\sim} \ \calD^{\bc}(\calR\Mod_-)^\op, \quad
\kappa_{\Omega}' :  \calD^{\bc}(\calT'\Mod_+) \ \xrightarrow{\sim} \ \calD^{\bc}(\calR'\Mod_-)^\op
\]
the associated equivalences of Theorem~\ref{thm:equiv2}.

Let $\varphi: \calX' \to \calX$ be a morphism of complexes. This morphism induces morphisms of $\Gm$-equivariant dg-algebras
\[
\Phi : \calT' \to \calT, \qquad \Psi : \calR \to \calR'.
\]
In turn, these morphisms of dg-algebras induce functors
\[
\Phi_* : \calC(\calT\Mod) \to \calC(\calT'\Mod), \quad \Psi_* : \calC(\calR'\Mod) \to \calC(\calR\Mod)
\]
(restriction of scalars) and
\[
\Phi^* : \calC(\calT'\Mod) \to \calC(\calT\Mod), \quad \Psi^* : \calC(\calR\Mod) \to \calC(\calR'\Mod)
\]
(extension of scalars).

The functors $\Phi_*$ and $\Psi_*$ are exact, hence induce functors
\[
R\Phi_* : \calD(\calT\Mod) \to \calD(\calT'\Mod), \quad R\Psi_* : \calD(\calR'\Mod) \to \calD(\calR\Mod).
\]
These functors clearly send the subcategory $\calD^\bc(\calT\Mod_\pm)$ into $\calD^\bc(\calT'\Mod_\pm)$ and the subcategory $\calD^\bc(\calR'\Mod_-)$ into $\calD^\bc(\calR\Mod_-)$ (and similarly without ``$\bc$'').

The functors $\Phi^*$ and $\Psi^*$ are not exact. However, it follows from \cite[Proposition 1.2.3]{MR} (existence of $K$-flat resolutions) that they admit left derived functors
\[
L\Phi^* : \calD(\calT'\Mod) \to \calD(\calT\Mod), \quad L\Psi^* : \calD(\calR\Mod) \to \calD(\calR'\Mod).
\]

The following result expresses the compatibility of our Koszul duality equivalence $\kappa_{\Omega}$ with morphisms of perfect sheaves. It is similar in spirit (but in a much more general setting) to \cite[Proposition 2.5.4]{R2}. 

\begin{prop}
\label{prop:kd-morphisms-complexes}
Let $X$ be a nice scheme with dualizing complex $\Omega$.

\begin{enumerate}
\item 
The functor $L\Psi^*$ restricts to a functor from $\calD^{\bc}(\calR\Mod_-)$ to $\calD^{\bc}(\calR'\Mod_-)$, denoted similarly. Moreover, there exists an isomorphism of functors from $\calD^\bc(\calT\Mod_+)$ to $\calD^\bc(\calR'\Mod_-)^\op$:
\[
L\Psi^* \circ \kappa_{\Omega} \ \cong \ \kappa'_{\Omega} \circ R\Phi_*.
\]
\item 
The functor $L\Phi^*$ restricts to a functor from $\calD^{\bc}(\calT'\Mod_+)$ to $\calD^{\bc}(\calT\Mod_+)$, denoted similarly. Moreover, there exists an isomorphism of functors from $\calD^\bc(\calT'\Mod_+)$ to $\calD^\bc(\calR\Mod_-)^\op$:
\[
\kappa_{\Omega} \circ L\Phi^* \ \cong \ R\Psi_* \circ \kappa_{\Omega}'.
\]
\end{enumerate}
\end{prop}

\subsection{Proof of Proposition~\ref{prop:kd-morphisms-complexes}}

To prove Proposition~\ref{prop:kd-morphisms-complexes} we need some preparatory lemmas. We assume the conditions in the proposition are satisfied.

\begin{lem}
\label{lem:S-K-flat}
The $\Gm$-equivariant $\calO_X$-dg-module $\calS$ is $K$-flat.
\end{lem}

\begin{proof}
It is enough to prove that for any $i \in \mathbb{Z}$, the $\calO_X$-dg-module $\calS_i$ is $K$-flat. However, this complex is a bounded complex of flat $\calO_X$-modules, which proves this fact.
\end{proof}

\begin{lem}
\label{lem:T-dg-mod-K-flat}
For every object $\calM$ of $\calC(\calT\Mod_-)$, there exists an object $\calM'$ of $\calC(\calT\Mod_-)$ which is $K$-flat as a $\Gm$-equivariant $\calO_X$-dg-module and such that the images of $\calM$ and $\calM'$ in $\calD(\calT\Mod_-)$ are isomorphic.
\end{lem}

\begin{proof}
Let $\calM$ be an object of $\calC(\calT\Mod_-)$. By Theorem~\ref{thm:Koszuldualitycomodules}, there exists an object $\calN$ of $\calC(\calS\Mod_-)$ such that $\calM$ and $\scrb(\calN)$ are isomorphic in the derived category. Hence we can assume that $\calM=\scrb(\calN)$.

By the same arguments as in \cite[Proposition 1.2.3]{MR}, using the existence of enough flat quasi-coherent $\calO_X$-modules in $\QCoh(X)$ (see \cite[\S 3.2]{BR}) one can check that for any object $\calN'$ of $\calC(\calS\Mod_-)$ there exists an object $\calN''$ of $\calC(\calS\Mod_-)$ which is $K$-flat as a $\Gm$-equivariant $\calS$-dg-module and a quasi-isomorphism $\calN'' \xrightarrow{\qis} \calN'$. Hence (as $\scrb$ is exact) we can assume that $\calN$ is $K$-flat as a $\Gm$-equivariant $\calS$-dg-module. Then it follows from Lemma~\ref{lem:S-K-flat} that $\calN$ is also $K$-flat as a $\Gm$-equivariant $\calO_X$-dg-module (see \cite[Lemma 1.3.2]{R2}).

We claim that, in this case, $\scrb(\calN)$ is $K$-flat as a $\Gm$-equivariant $\calO_X$-dg-module. Indeed, it is enough to show that for any $n \in \mathbb{Z}$, $\scrb(\calN)_n$ is a $K$-flat $\calO_X$-dg-module. However, we have
\[
\scrb(\calN)_n = \bigoplus_{k+l=n} (\calT^{\vee})_k \otimes_{\calO_X} \calN_l.
\]
This sum is finite and, as in the proof of Theorem~\ref{thm:Koszuldualitycomodules}, this dg-module can be obtained from the dg-modules $(\calT^{\vee})_k \otimes_{\calO_X} \calN_l$ (ordinary tensor product) by taking shifts and cones a finite number of times. The latter dg-modules are $K$-flat over $\calO_X$, and the cone of a morphism between $K$-flat dg-modules is still $K$-flat, hence these remarks finish the proof of our claim, and also of the lemma.
\end{proof}

\begin{lem}
\label{lem:image-C-K-flat}
If $\calM$ is an object of $\calC(\calT\Mod_-)$ which is $K$-flat as a $\Gm$-equivariant $\calO_X$-dg-module, then $\scra(\calM)$ is $K$-flat as a $\Gm$-equivariant $\calS$-dg-module.
\end{lem}

\begin{proof}
We have to check that for any acyclic object $\calN$ of $\calC(\calS\Mod)$, the complex $\calN \otimes_{\calS} \scra(\calM)$ is acyclic. Every object of $\calC(\calS\Mod)$ is a direct limit of objects of $\calC(\calS\Mod_-)$ (because $\calS$ is concentrated in non-positive internal degrees); moreover if the initial object is acyclic one can choose these objects to be also acyclic. Hence we can assume that $\calN$ is in $\calC(\calS\Mod_-)$. Then we have $\calN \otimes_{\calS} \scra(\calM) = \calN \otimes_{\calO_X} \calM$, where the differential is the sum of the usual differential of the tensor product $\calN \otimes_{\calO_X} \calM$ and a Koszul-type differential. The same argument as in the proof of Theorem~\ref{thm:Koszuldualitycomodules} or Lemma~\ref{lem:T-dg-mod-K-flat} proves that this complex is acyclic.
\end{proof}

\begin{proof}[of Proposition~{\rm \ref{prop:kd-morphisms-complexes}}]
We only prove $\rmi$; the proof of $\rmii$ is similar.

The regrading functors of \S \ref{ss:kd-version3} do not play any role here, hence we will rather work with the equivalences $K_{\Omega}$, $K'_{\Omega}$ of Theorem \ref{thm:equiv}. For simplicity, we still denote by $L\Psi^*$ the derived extension of scalars from $\calS$- to $\calS'$-dg-modules. We denote by $\scra'$, $\scrb'$ the functors of \S\ref{ss:definitions} relative to the complex $\calX'$. It is clear from definition that we have an isomorphism of functors
\begin{equation}
\label{eqn:isom-functors-morphism-1}
\mathbf{D}_{\Omega}^{\calT'} \circ R\Phi_* \ \cong \ R\Phi_* \circ \mathbf{D}_{\Omega}^{\calT}.
\end{equation}
(Here, the first $R\Phi_*$ is considered as a functor from $\calD^{\bc}(\calT\Mod_+)$ to $\calD^{\bc}(\calT'\Mod_+)$, while the second one is considered as a functor from $\calD^{\bc}(\calT\Mod_-)$ to $\calD^{\bc}(\calT'\Mod_-)$.)

We claim that there exists an isomorphism of functors from $\calD(\calT\Mod_-)$ to $\calD(\calS'\Mod)$
\begin{equation}
\label{eqn:isom-functors-morphism-2}
\overline{\scra'} \circ R\Phi_* \ \cong \ L\Psi^* \circ \overline{\scra}.
\end{equation}
First, consider the assigment from $\calC(\calT\Mod_-)$ to $\calC(\calS'\Mod)$ given by $\calM \mapsto \Psi^* \scra(\calM) \cong \calS' \otimes_{\calO_X} \calM \cong \scra' \Phi_* (\calM)$. (Here the differential on $\calS' \otimes_{\calO_X} \calM$ involves a Koszul-type differential as usual.) This functor is exact, as the composition of the exact functors $\Phi_*$ and $\scra'$. The induced functor from $\calD(\calT\Mod_-)$ to $\calD(\calS'\Mod)$ is clearly isomorphic to the left hand side of \eqref{eqn:isom-functors-morphism-2}.

By usual properties of composition of derived functors, we obtain a morphism of functors
\begin{equation}
\label{eqn:morphism-functors-Phi-Psi}
L\Psi^* \circ \overline{\scra} \to \overline{\scra'} \circ R\Phi_*.
\end{equation}
What we have to show is that this morphism is an isomorphism. By Lemma \ref{lem:T-dg-mod-K-flat}, it is enough to show that it is an isomorphism when applied to an object $\calM$ of $\calC(\calT\Mod_-)$ which is $K$-flat as a $\Gm$-equivariant $\calO_X$-dg-module. However, in this case $\scra(\calM)$ is $K$-flat over $\calS$ by Lemma \ref{lem:image-C-K-flat}, hence $L\Psi^* \circ \overline{\scra}(\calM)$ is the image in the derived category of $\Psi^* \circ \scra(\calM)$, which finishes the proof of \eqref{eqn:isom-functors-morphism-2}.

Composing \eqref{eqn:isom-functors-morphism-1} with the restriction of \eqref{eqn:isom-functors-morphism-2} and then with the regrading functor $\xi'$ we obtain an isomorphism of functors from $\calD^{\bc}(\calT\Mod_+)$ to $\calD(\calR'\Mod)^\op$
\begin{equation}
\label{eqn:isom-functors-morphism-3}
L\Psi^* \circ \kappa_{\Omega} \ \cong \ \kappa'_{\Omega} \circ R\Phi_*.
\end{equation}
As the right hand side sends $\calD^{\bc}(\calT\Mod_+)$ to $\calD^{\bc}(\calR'\Mod_-)^\op$, and as $\kappa_{\Omega}$ is an equivalence between $\calD^{\bc}(\calT\Mod_+)$ and $\calD^{\bc}(\calR\Mod_-)^\op$, it follows from \eqref{eqn:isom-functors-morphism-3} that $L\Psi^*$ restricts to a functor from $\calD^{\bc}(\calR\Mod_-)$ to $\calD^{\bc}(\calR'\Mod_-)$. Then \eqref{eqn:isom-functors-morphism-3} proves the isomorphism of $\rmi$, hence finishes the proof.
\end{proof}

\subsection{Application to intersection of subbundles}

Now we will explain the geometric content of Proposition \ref{prop:kd-morphisms-complexes} in the context of \S\ref{ss:lkd}. We let $E$ and $E'$ be vector bundles on $X$, and let
\[
\xymatrix@R=0.3cm{
E \ar[rd] \ar[rr]^{\phi} & & E' \ar[ld] \\
& X &
}
\]
be a morphism of vector bundles over $X$. We consider subbundles $F_1, F_2 \subseteq E$ and $F_1', F_2' \subseteq E'$, and assume that
\[
\phi(F_1) \subseteq F_1', \qquad \phi(F_2) \subseteq F_2'.
\]
Let $\calE, \, \calF_1, \, \calF_2, \, \calE', \, \calF_1', \, \calF_2'$ be the respective sheaves of sections of $E, \, F_1, \, F_2, \, E', \, F_1', \, F_2'$. By Theorem \ref{thm:lkd} we have linear Koszul duality equivalences
\begin{align*}
\kappa_{\Omega} : \calD^{\mathrm{c}}_{\Gm}(F_1 \, \rcap_E \, F_2) \ & \xrightarrow{\sim} \ \calD^{\mathrm{c}}_{\Gm}(F_1^{\bot} \, \rcap_{E^*} \, F_2^{\bot})^\op, \\
\kappa_{\Omega}' : \calD^{\mathrm{c}}_{\Gm}(F_1' \, \rcap_{E'} \, F_2') \ & \xrightarrow{\sim} \ \calD^{\mathrm{c}}_{\Gm} \bigl( (F_1')^{\bot} \, \rcap_{(E')^*} \, (F_2')^{\bot} \bigr)^\op.
\end{align*}

We consider the complexes $\calX$ (for the vector bundle $E$) and $\calX'$ (for the vector bundle $E'$) defined as in \S \ref{ss:lkd}. The morphism $\phi$ defines a morphism of complexes $\calX' \to \calX$, to which we can apply the constructions of \S \ref{ss:kd-morphisms-complexes}.

More geometrically, $\phi$ induces a morphism of dg-schemes $\Phi : F_1 \, \rcap_E \, F_2 \to F_1' \, \rcap_{E'} \, F_2'$, and we have a (derived) direct image functor
\[
R\Phi_* : \calD^{\mathrm{c}}_{\Gm}(F_1 \, \rcap_E \, F_2) \to \calD(\calT'\Mod).
\]
This functor is just the restriction of the functor denoted similarly in \S \ref{ss:kd-morphisms-complexes} (in our special case). 

\begin{lem}
\label{lem:direct-image-Phi}
Assume that the induced morphism of schemes between non-derived intersections $F_1 \cap_E F_2 \to F_1' \cap_{E'} F_2'$ is proper. Then the functor $R\Phi_*$ sends $\calD^{\mathrm{c}}_{\Gm}(F_1 \, \rcap_E \, F_2)$ into $\calD^{\mathrm{c}}_{\Gm}(F_1' \, \rcap_{E'} \, F_2')$. 
\end{lem}

\begin{proof}
The dg-algebras $\calT$ and $\calT'$ are both concentrated in non-positive cohomological degrees; hence there exist natural morphisms of dg-algebras $\Theta : \calT \to \calH^0(\calT)$, $\Theta' : \calT' \to \calH^0(\calT')$. Let also $\Phi^0 : \calH^0(\calT') \to \calH^0(\calT)$ be the morphism induced by $\Phi$, so that we have $\Phi^0 \circ \Theta' = \Theta \circ \Phi$. Taking direct images (i.e.~restriction of scalars) we obtain a commutative diagram
\[
\xymatrix@C=2cm{
\calD \bigl( \calH^0(\calT)\Mod \bigr) \ar[r]^-{R\Phi^0_*} \ar[d]_-{\Theta_*} & \calD \bigl( \calH^0(\calT')\Mod \bigr) \ar[d]^-{\Theta'_*} \\
\calD(\calT\Mod) \ar[r]^-{R\Phi_*} & \calD(\calT'\Mod).
}
\]

One can easily check that the functor $\Theta_*$ restricts to a functor from $\calD^\fg \bigl( \calH^0(\calT)\Mod \bigr)$ to $\calD^\fg(\calT\Mod)$, and similarly for $\Theta'_*$. Moreover, 
using truncation functors as in~\S\ref{ss:finiteness},
one can check that the essential image of $\Theta_*$ generates the category $\calD^\fg(\calT\Mod)$.
Hence it is enough to prove that $R\Phi^0_*$ sends $\calD^\fg \bigl( \calH^0(\calT)\Mod \bigr)$ into $\calD^\fg \bigl( \calH^0(\calT')\Mod \bigr)$. However the morphism from $F_1 \cap_E F_2$ to $X$ is affine, and the direct image of the structure sheaf under this morphism is $\calH^0(\calT)$, so that we obtain an equivalence of categories
\[
\calD^\fg \bigl( \calH^0(\calT)\Mod \bigr) \ \cong \ \calD^b\Coh^{\Gm}(F_1 \cap_E F_2)
\]
where $t \in \Gm$ acts on $E$ by dilatation by $t^{-2}$ along the fibers, and on $F_1 \cap_E F_2$ by restriction.
Similarly we have an equivalence
\[
\calD \bigl( \calH^0(\calT')\Mod \bigr) \ \cong \ \calD \QCoh^{\Gm}(F_1' \cap_E F_2').
\]
and under these equivalences the functor $R\Phi^0_*$ identifies with the (derived) direct image under the morphism $F_1 \cap_E F_2 \to F_1' \cap_{E'} F_2'$. Hence our claim follows from \cite[Proposition II.2.2]{H}.
\end{proof}

We also consider the (derived) inverse image functor
\[
L\Phi^* : \calD^{\mathrm{c}}_{\Gm}(F_1' \, \rcap_{E'} \, F_2') \to \calD(\calT\Mod).
\]
Again, this functor is the restriction of the functor denoted similarly in \S \ref{ss:kd-morphisms-complexes}.

The morphism $\phi$ induces a morphism of vector bundles
\[
\psi:=\phi^{\vee} : (E')^* \to E^*,
\]
which satisfies $\psi((F_i')^{\bot}) \subset F_i^{\bot}$ for $i=1,2$. Hence the above constructions and results also apply to $\psi$. We use similar notation.

The following result is an immediate application of Proposition \ref{prop:kd-morphisms-complexes} and Lemma \ref{lem:direct-image-Phi}.

\begin{prop}
\label{prop:morphisms}
Assume that $X$ is a nice scheme with dualizing complex $\Omega$.

\begin{enumerate}
\item 
Assume that the morphism of schemes $F_1 \cap_E F_2 \to F_1' \cap_{E'} F_2'$ induced by $\phi$ is proper. Then $L\Psi^*$ sends the subcategory $\calD^{\mathrm{c}}_{\Gm}(F_1^{\bot} \, \rcap_{E^*} \, F_2^{\bot})$ into $\calD^{\mathrm{c}}_{\Gm}((F_1')^{\bot} \, \rcap_{(E')^*} \, (F_2')^{\bot})$. Moreover, there exists a natural isomorphism of functors from $\calD^{\mathrm{c}}_{\Gm}(F_1 \, \rcap_{E} \, F_2)$ to $\calD^{\mathrm{c}}_{\Gm} \bigl( (F_1')^{\bot} \, \rcap_{(E')^*} \, (F_2')^{\bot} \bigr)^\op$:
\[
L\Psi^* \circ \kappa_{\Omega} \ \cong \ \kappa'_{\Omega} \circ R\Phi_*.
\]
\item 
Assume that the morphism of schemes $(F_1')^{\bot} \cap_{(E')^*} (F_2')^{\bot} \to F_1^{\bot} \cap_{E'} F_2^{\bot}$ induced by $\psi$ is proper. Then $L\Phi^*$ sends $\calD^{\mathrm{c}}_{\Gm}(F_1' \, \rcap_{E'} \, F_2')$ into $\calD^{\mathrm{c}}_{\Gm}(F_1 \, \rcap_E \, F_2)$. Moreover, there exists a natural isomorphism of functors from $\calD^{\mathrm{c}}_{\Gm}(F_1' \, \rcap_{E'} \, F_2')$ to $\calD^{\mathrm{c}}_{\Gm}(F_1^{\bot} \, \rcap_{E^*} \, F_2^{\bot})^\op$:
\[
\kappa_{\Omega} \circ L\Phi^* \ \cong \ R\Psi_* \circ \kappa'_{\Omega}.
\]
\end{enumerate}

In particular, if both assumptions are satisfied, the following diagram is commutative:
\[
\xymatrix@C=3cm{
\calD^{\mathrm{c}}_{\Gm}(F_1 \, \rcap_E \, F_2) \ar[r]^-{\kappa_{\Omega}}_-{\sim} \ar@<0.5ex>[d]^-{R\Phi_*} & \calD^{\mathrm{c}}_{\Gm}(F_1^{\bot} \, \rcap_{E^*} \, F_2^{\bot})^\op \ar@<0.5ex>[d]^-{L\Psi^*} \\
\calD^{\mathrm{c}}_{\Gm}(F_1' \, \rcap_{E'} \, F_2') \ar@<0.5ex>[u]^-{L\Phi^*} \ar[r]^-{\kappa'_{\Omega}}_-{\sim} & \calD^{\mathrm{c}}_{\Gm} \bigl( (F_1')^{\bot} \, \rcap_{(E')^*} \, (F_2')^{\bot} \bigr)^\op. \ar@<0.5ex>[u]^-{R\Psi_*}
}
\]
\end{prop}

\section{Linear Koszul duality and base change}
\label{sec:basechange}

\subsection{Statement}
\label{ss:defbasechange}

Let us come back to the setting of \S \ref{ss:kd-version3}. More precisely, let $X$ and $Y$ be nice schemes, and let $\pi: X \to Y$ be a morphism of finite type. Let $\Omega$ be a dualizing complex for $Y$; then by \cite[Corollary VI.3.5]{H} $\pi^! \Omega$ is a dualizing complex for $X$.

Let $\calX$ be a complex of locally free sheaves on $Y$ of the form considered in \S \ref{ss:definitions}, and let $\calR_Y,\calT_Y,\calS_Y$ be the associated dg-algebras. We also consider the complex of locally free sheaves $\pi^* \calX$ on $X$, and let $\calR_X,\calT_X,\calS_X$ be the associated dg-algebras. Note that we have natural isomorphisms of dg-algebras
\[
\calR_X \cong \pi^* \calR_Y, \quad \calS_X \cong \pi^* \calS_Y, \quad \calT_X \cong \pi^* \calT_Y.
\]
We denote by
\[
\kappa_{\Omega}^Y : \calD^{\bc}(\calT_Y\Mod_+) \xrightarrow{\sim} \calD^{\bc}(\calR_Y\Mod_-)^\op, \quad \kappa_{\pi^! \Omega}^X : \calD^{\bc}(\calT_X\Mod_+) \xrightarrow{\sim} \calD^{\bc}(\calR_X\Mod_-)^\op
\]
the associated equivalences of Theorem \ref{thm:equiv2}.

The morphism of schemes $\pi$ induces a morphism of dg-ringed spaces $\hat{\pi} : (X, \, \calT_X) \to (Y, \, \calT_Y)$, and we can consider the associated direct and inverse image functors $\hat{\pi}_* : \calC(\calT_X\Mod) \to \calC(\calT_Y\Mod)$, $\hat{\pi}^* : \calC(\calT_Y\Mod) \to \calC(\calT_X\Mod)$. Note that the following diagram commutes:
\begin{equation}
\label{eqn:diagram-pi-hat}
\vcenter{
\xymatrix@C=2cm{
\calC(\calT_Y\Mod) \ar[r]^-{\hat{\pi}^*} \ar[d]_-{\For} & \calC(\calT_X\Mod) \ar[d]^-{\For} \\
\calC(\calO_Y\Mod) \ar[r]^-{\pi^*} & \calC(\calO_X\Mod).
}
}
\end{equation}
One can also consider the associated derived functors
\[
R\hat{\pi}_* : \calD(\calT_X\Mod) \to \calD(\calT_Y\Mod), \qquad L\hat{\pi}^* : \calD(\calT_Y\Mod) \to \calD(\calT_X\Mod)
\]
(see \S\ref{ss:nice-schemes}).

In the following lemma we will say $\pi$ has finite Tor-dimension if for any $\calF$ in  $\QCoh(Y)$ the set $\{i \in \Z \mid \calH^i(L\pi^* \calF) \neq 0 \}$ is bounded. (Note that this terminology is different from the one used in \cite{H}!)

\begin{lem}
\label{lem:base-change-derived-functors}
\begin{enumerate}
\item Assume that $\pi$ has finite Tor-dimension. Then the functor $L\hat{\pi}^*$ restricts to a functor from $\calD^{\bc}(\calT_Y\Mod_{\pm})$ to $\calD^{\bc}(\calT_X\Mod_{\pm})$.
\item Assume that $\pi$ is proper. Then the functor $R\hat{\pi}_*$ restricts to a functor from $\calD^{\bc}(\calT_X\Mod_+)$ to $\calD^{\bc}(\calT_Y\Mod_+)$.
\end{enumerate}
\end{lem}

\begin{proof} 
$\rmi$ The dg-algebra $\calT_Y$ is $K$-flat as a $\Gm$-equivariant $\calO_X$-dg-module, hence any $K$-flat $\Gm$-equivariant $\calT_Y$-dg-module is also $K$-flat as a $\Gm$-equivariant $\calO_X$-dg-module (see \cite[Lemma 1.3.2$\rmii$]{R2}). It follows (using \eqref{eqn:diagram-pi-hat}) that the following diagram commutes:
\[
\xymatrix@C=2cm{
\calD(\calT_Y\Mod) \ar[r]^-{L\hat{\pi}^*} \ar[d]_-{\For} & \calD(\calT_X\Mod) \ar[d]^-{\For} \\
\calD(\calO_Y\Mod) \ar[r]^-{L\pi^*} & \calD(\calO_X\Mod).
}
\]
Hence we only have to prove the result when $\calX=0$, in which case it follows from \cite[Proposition II.4.4]{H} and our assumption on $\pi$.

The proof of $\rmii$ is similar, using the compatibility of the functor $R\hat{\pi}_*$ with $R\pi_*$ (see \cite[Proposition 3.3.6]{BR}) and \cite[Proposition II.2.2]{H}.
\end{proof}

In addition to the above functors, we will also consider the following one, in case $\pi$ has finite Tor-dimension:
\[
\hat{\pi}^! := \mathbf{D}_{\pi^! \Omega}^{\calT_X} \circ L\hat{\pi}^* \circ \mathbf{D}_{\Omega}^{\calT_Y} : \calD^{\bc}(\calT_Y\Mod_+) \to \calD^{\bc}(\calT_X\Mod_+)
\]
(see \S \ref{ss:duality} for the definition of Grothendieck--Serre duality functors).

Similarly, $\pi$ induces a morphism of dg-ringed spaces $\tilde{\pi} : (X,\calR_X) \to (Y,\calR_Y)$. By the same arguments as above, if $\pi$ has finite Tor-dimension then we have a derived inverse image functor
\[
L\tilde{\pi}^* : \calD^{\bc}(\calR_Y\Mod_-) \to \calD^{\bc}(\calR_X\Mod_-)
\]
and, if $\pi$ is proper, we have a derived direct image functor
\[
R\tilde{\pi}_* : \calD^{\bc}(\calR_X\Mod_-) \to \calD^{\bc}(\calR_Y\Mod_-).
\]

The following result expresses the compatibility of our Koszul duality equivalence with base change. It is similar in spirit to \cite[Proposition 2.4.5]{R2}.

\begin{prop}
\label{prop:compatibility-kd-base-change}
Assume that $X$ and $Y$ are nice schemes and that $\pi$ is of finite type. Let $\Omega$ be a dualizing complex for $Y$.

\begin{enumerate}
\item 
Assume moreover that $\pi$ has finite Tor-dimension. Then there exists an isomorphism
\[
L\tilde{\pi}^* \circ \kappa_{\Omega}^Y \ \cong \ \kappa_{\pi^! \Omega}^X \circ \hat{\pi}^!
\]
of functors from $\calD^{\bc}(\calT_Y\Mod_+)$ to $\calD^{\bc}(\calR_X\Mod_-)^\op$
\item 
Assume moreover that $\pi$ is proper. Then there exists an isomorphism
\[
R\tilde{\pi}_* \circ \kappa_{\pi^! \Omega}^X \ \cong \ \kappa_{\Omega}^Y \circ R\hat{\pi}_*
\]
of functors from $\calD^{\bc}(\calT_X\Mod_+)$ to $\calD^{\bc}(\calR_Y\Mod_-)^\op$.
\end{enumerate}
\end{prop}

\subsection{Proof of Proposition~\ref{prop:compatibility-kd-base-change}}
\label{ss:proof-base-change}

We begin by proving some compatibility results for the equivalences $\overline{\scra}$ and $\overline{\scrb}$ of \S \ref{ss:KD1} (under weaker assumptions than in the proposition). We use subscripts to indicate the scheme we are working on. The morphism $\pi$ induces a morphism of dg-ringed spaces $\check{\pi} : (X,\calS_X) \to (Y,\calS_Y)$. As above, we have corresponding derived direct and inverse image functors\footnote{Note that $\calS_X$ and $\calS_Y$ are not non-positively graded, so that we cannot apply directly the results quoted in \S\ref{ss:nice-schemes}. However, one can reduce our situation to the one of \S\ref{ss:nice-schemes} using the regrading equivalence of \S\ref{ss:kd-version3}.}
\[
R\check{\pi}_* : \calD(\calS_X\Mod) \to \calD(\calS_Y\Mod), \qquad L\check{\pi}^* : \calD(\calS_Y\Mod) \to \calD(\calS_X\Mod).
\]
Using the same arguments as in the proof of Lemma \ref{lem:base-change-derived-functors}, one can check that these functors restrict to functors
\[
R\check{\pi}_* : \calD(\calS_X\Mod_-) \to \calD(\calS_Y\Mod_-), \qquad L\check{\pi}^* : \calD(\calS_Y\Mod_-) \to \calD(\calS_X\Mod_-)
\]
which form an adjoint pair, and that the functors $R\hat{\pi}_*$ and $L\hat{\pi}^*$ restrict to functors
\[
R\hat{\pi}_* : \calD(\calT_X\Mod_-) \to \calD(\calT_Y\Mod_-), \qquad L\hat{\pi}^* : \calD(\calT_Y\Mod_-) \to \calD(\calT_X\Mod_-)
\]
which also form an adjoint pair.

\begin{prop}
\label{prop:compatibility-C-D-base-change}
Assume that $X$ and $Y$ are nice schemes.

\begin{enumerate}
\item There exists an isomorphism of functors from $\calD(\calT_Y\Mod_-)$ to $\calD(\calS_X\Mod_-)$:
\[
L\check{\pi}^* \circ \overline{\scra}_Y \ \cong \ \overline{\scra}_X \circ L\hat{\pi}^*.
\]
\item There exists an isomorphism of functors from $\calD(\calT_X\Mod_-)$ to $\calD(\calS_Y\Mod_-)$:
\[
\overline{\scra}_Y \circ R\hat{\pi}_* \ \cong \ R\check{\pi}_* \circ \overline{\scra}_X.
\]
\end{enumerate}
\end{prop}

\begin{proof}
$\rmi$ By Theorem \ref{thm:Koszuldualitycomodules}, it is equivalent to prove an isomorphism of functors from $\calD(\calS_Y\Mod_-)$ to $\calD(\calT_X\Mod_-)$:
\[
\overline{\scrb}_X \circ L\check{\pi}^* \ \cong \ L\hat{\pi}^* \circ \overline{\scrb}_Y.
\]
By construction of resolutions, there are enough objects in $\calC(\calS_Y\Mod_-)$ which are $K$-flat as $\Gm$-equivariant $\calS_Y$-dg-modules.
Using this observation and the same arguments as for \eqref{eqn:morphism-functors-Phi-Psi} one can easily construct a morphism of functors from the right hand side to the left hand side; what remains is to prove that it is an isomorphism.
For this it is enough to prove that for any object $\calN$ of $\calC(\calS_Y\Mod_-)$ which is $K$-flat as a $\Gm$-equivariant $\calS_Y$-dg-module, with image $\overline{\calN}$ in the derived category, our morphism
\begin{equation}
\label{eq:isom-D-inverse-image}
L\hat{\pi}^* \circ \overline{\scrb}_Y(\overline{\calN}) \ \to \ \overline{\scrb}_X \circ L\check{\pi}^*(\overline{\calN})
\end{equation}
is an isomorphism.

The right hand side of \eqref{eq:isom-D-inverse-image} is easy to compute: it is isomorphic to the image in the derived category of the $\calT_X$-dg-module $\scrb_X \bigl( \check{\pi}^* \calN \bigr)$. Let us consider now the left hand side. As $\calS_Y$ is $K$-flat as a $\Gm$-equivariant $\calO_Y$-dg-module (see Lemma \ref{lem:S-K-flat}), any $K$-flat $\Gm$-equivariant $\calS_Y$-dg-module is also $K$-flat as a $\Gm$-equivariant $\calO_Y$-dg-module (see \cite[Lemma 1.3.2]{R2}); in particular $\calN$ has this property. As checked in the course of the proof of Lemma \ref{lem:T-dg-mod-K-flat}, this implies that $\scrb_Y (\calN)$ is $K$-flat as a $\Gm$-equivariant $\calO_Y$-dg-module. From this it follows that the left hand side of \eqref{eq:isom-D-inverse-image} is isomorphic to the image in the derived category of $\hat{\pi}^* \bigl( \scrb_Y(\calN) \bigr)$. Indeed, let $\calP$ be a $K$-flat resolution of $\scrb_Y(\calN)$ as a $\Gm$-equivariant $\calT_Y$-dg-module; then the morphism $\hat{\pi}^* \calP \to \hat{\pi}^* \bigl( \scrb_Y(\calN) \bigr)$ is a quasi-isomorphism since the functor $\hat{\pi}^*$ sends every acyclic $\calT_Y$-dg-module which is $K$-flat as a $\Gm$-equivariant $\calO_Y$-dg-module to an acyclic dg-module.

Using these remarks, the fact that \eqref{eq:isom-D-inverse-image} is an isomorphism now follows from the natural isomorphism $\scrb_X \bigl( \check{\pi}^* \calN \bigr) \cong \hat{\pi}^* \bigl( \scrb_Y(\calN) \bigr)$, which finishes the proof of $\rmi$.

Isomorphism $\rmii$ follows from $\rmi$ by adjunction.
\end{proof}

We will also need the following compatibility result for direct images and duality. It is a ``dg-version'' of Hartshorne's Duality Theorem of \cite{H}.

\begin{prop}
\label{prop:compatibility-direct-image-duality}
Assume that $X$ and $Y$ are nice schemes, and that $\pi$ is proper and of finite type. Let $\Omega$ be a dualizing complex for $Y$. Then there exists an isomorphism of functors from $\calD^{\bc}(\calT_X\Mod)$ to $\calD^{\bc}(\calT_Y\Mod)^\op$:
\[
R\hat{\pi}_* \circ \mathbf{D}_{\pi^! \Omega}^{\calT_X} \ \cong \ \mathbf{D}_{\Omega}^{\calT_Y} \circ R\hat{\pi}_*.
\]
\end{prop}

\begin{proof}
In this proof we will work with some dg-modules which are not necessarily quasi-coherent, but which have quasi-coherent cohomology. This is allowed by \cite[Proposition 3.3.2]{BR}. Note also that by \cite[Proposition 3.3.6]{BR} we can compute the image under $R\hat{\pi}_*$ of an object of $\calD(\calT_X\Mod)$ in the larger derived category of all sheaves of $\calT_X$-dg-modules. Finally, in this proof we write $Q \calM$ for the image in the suitable derived category of a dg-module $\calM$.

Recall that we have fixed bounded below complexes of injective quasi-coherent $\calO_X$-modules $\calI_{\Omega}$ and $\calI_{\pi^! \Omega}$ such that $Q\calI_{\Omega} \cong \Omega$ and $Q \calI_{\pi^! \Omega} \cong \pi^! \Omega$. By \cite[p.~257]{H}, we can assume that both of them are bounded complexes. Note that, as $X$ and $Y$ are in particular locally Noetherian, the terms of these complexes are injective as $\calO$-modules (see \cite[Proposition II.7.17]{H}).

Start with some $\calM$ in $\calC(\calT_X\Mod)$ such that $Q\calM$ is in $\calD^{\bc}(\calT_X\Mod)$, and denote by $\calJ$ a $K$-injective object of $\calC(\calT_X\Mod)$ endowed with a quasi-isomorphism $\calM \xrightarrow{\qis} \calJ$. Then we have 
\begin{equation}
\label{eq:direct-image-duality-1}
R\hat{\pi}_* \circ \mathbf{D}_{\pi^! \Omega}^{\calT_X}(Q \calM) \ \cong \ Q \hat{\pi}_* \sheafHom_{\calO_X}(\calJ,\calI_{\pi^! \Omega}).
\end{equation}
Indeed, by definition we have $\mathbf{D}_{\pi^! \Omega}^{\calT_X}(Q \calM) \cong Q \sheafHom_{\calO_X}(\calM,\calI_{\pi^! \Omega})$. Now as $\calI_{\pi^! \Omega}$ is a bounded complex of injective $\calO_X$-modules, it is $K$-injective in the category of all $\calO_X$-modules, hence the morphism $\sheafHom_{\calO_X}(\calJ,\calI_{\pi^! \Omega}) \to \sheafHom_{\calO_X}(\calM,\calI_{\pi^! \Omega})$ is a quasi-isomorphism (see \cite[Lemma (2.4.5.1)]{Li}). Now again as $\calI_{\pi^! \Omega}$ is a bounded complex of injective $\calO_X$-module, the complex of $\calO_X$-modules $\sheafHom_{\calO_X}(\calJ,\calI_{\pi^! \Omega})$ is made of flasque sheaves, which are $\pi_*$-acyclic. By \cite[Corollary (3.9.3.5)]{Li}, we deduce that the natural morphism $Q \pi_* \sheafHom_{\calO_X}(\calJ,\calI_{\pi^! \Omega}) \to R\pi_* Q \sheafHom_{\calO_X}(\calJ,\calI_{\pi^! \Omega})$ is an isomorphism. 
This proves \eqref{eq:direct-image-duality-1}.

Now there exists a natural morphism of dg-modules
\[
\hat{\pi}_* \sheafHom_{\calO_X}(\calJ,\calI_{\pi^! \Omega}) \to \sheafHom_{\calO_Y}(\hat{\pi}_* \calJ,\pi_*\calI_{\pi^! \Omega}).
\]
Composing with the morphism induced by the ``trace morphism" $\pi_*\calI_{\pi^! \Omega} \to \calI_{\Omega}$ one obtains a morphism of dg-modules
\[
\hat{\pi}_* \sheafHom_{\calO_Y}(\calJ,\calI_{\pi^! \Omega}) \to \sheafHom_{\calO_Y}(\hat{\pi}_* \calJ,\calI_{\Omega}).
\]
Moreover, the image in the derived category of the right hand side is isomorphic to $\mathbf{D}_{\Omega}^{\calT_Y} \circ R\hat{\pi}_*(Q\calM)$. Hence we have constructed a natural morphism
\[
R\hat{\pi}_* \circ \mathbf{D}_{\pi^! \Omega}^{\calT_X} (Q\calM)\ \to \ \mathbf{D}_{\Omega}^{\calT_Y} \circ R\hat{\pi}_*(Q\calM).
\]
After forgetting the action of $\calT_Y$, this is the ``duality morphism" of \cite[p.~378]{H}, which is known to be an isomorphism (see \cite[Theorem VII.3.3]{H}). Hence this morphism is also an isomorphism, which finishes the proof of the proposition.
\end{proof}

\begin{proof}[of Proposition~{\rm \ref{prop:compatibility-kd-base-change}}]
$\rmi$ We obtain the isomorphism as the following composition of isomorphisms of functors:
\[
\begin{array}{rcl}
L\tilde{\pi}^* \circ \kappa_\Omega^Y & = & L\tilde{\pi}^* \circ \xi_Y \circ \overline{\scra}_Y^\bc \circ \mathbf{D}_\Omega^{\calT_Y} \\
& \cong & \xi_X \circ L{\check \pi}^* \circ \overline{\scra}_Y^\bc \circ \mathbf{D}_\Omega^{\calT_Y} \\
& \overset{\text{Prop.~\ref{prop:compatibility-C-D-base-change}}}{\cong} & \xi_X \circ \overline{\scra}_X^\bc \circ L\hat{\pi}^* \circ \mathbf{D}_\Omega^{\calT_Y} \\
& = & \kappa^X_{\pi^! \Omega} \circ \hat{\pi}^!.
\end{array}
\]

$\rmii$ We obtain the isomorphism as the following composition of isomorphisms of functors:
\[
\begin{array}{rcl}
R\tilde{\pi}_* \circ \kappa^X_{\pi^! \Omega} & = & R\tilde{\pi}_* \circ \xi_X \circ \overline{\scra}^\bc_X \circ \mathbf{D}_{\pi^! \Omega}^{\calT_X} \\
& \cong & \xi_Y \circ R{\check \pi}_* \circ \overline{\scra}^\bc_X \circ \mathbf{D}_{\pi^! \Omega}^{\calT_X} \\
& \overset{\text{Prop.~\ref{prop:compatibility-C-D-base-change}}}{\cong} & \xi_Y \circ \overline{\scra}^\bc_Y \circ R\hat{\pi}_* \circ \mathbf{D}_{\pi^! \Omega}^{\calT_X} \\
& \overset{\text{Prop.~\ref{prop:compatibility-direct-image-duality}}}{\cong} & \kappa^Y_\Omega \circ R\hat{\pi}_*.
\end{array}
\]
\end{proof}

\subsection{Compatibility with inverse image in the case of smooth varieties}

The isomorphism of Proposition~\ref{prop:compatibility-kd-base-change}$\rmi$ is not easy to work with in general since the functor $\hat{\pi}^!$ does not have a completely explicit description. In this subsection we give an easier description of this functor in the case $X$ and $Y$ are Noetherian, integral, separated, regular schemes of finite dimension. In this case $\pi$ has automatically finite Tor-dimension, as follows from \cite[Ex.~III.6.9]{HAG}. Moreover $\Omega$ is a shift of a line bundle (see \cite[Theorem V.3.1 and \S V.10]{H}), hence $\pi^* \Omega$ is also a dualizing complex on $X$. For simplicity we will also assume that $n \leq 1$ and work with locally finitely generated dg-modules, so that we are in the setting of \S\ref{ss:finiteness}.

\begin{lem}
\label{lem:smooth-case}
Assume that $Y$ is a Noetherian, integral, separated, regular scheme of finite dimension.
For any $\calM$ in $\calD^\fg(\calT_Y\Mod)$, there exists an object $\calP$ in $\calC(\calT_Y\Mod)$ such that $\calP_j$ is a finite complex of locally free $\calO_Y$-modules of finite rank for any $j \in \Z$, and whose image in $\calD(\calT_Y\Mod)$ is isomorphic to $\calM$.
\end{lem}

\begin{proof}
Using truncation functors we can assume that $\calM$ is bounded (for the cohomological grading), and then
by~\cite[Lemma~3.6.2]{MR} we can assume that $\calM$ is a locally finitely generated $\calT_Y$-dg-module. Then the construction in the proof of \cite[Proposition 3.1.1]{MR} produces an object $\calP$ as in the statement and a quasi-isomorphism $\calP \xrightarrow{\qis} \calM$.
\end{proof}

\begin{prop}
\label{prop:!-smooth}
Assume that $X$ and $Y$ are Noetherian, integral, separated, regular schemes of finite dimension. Then there exists an isomorphism of functors from $\calD^\fg(\calT_Y\Mod)$ to $\calD^\bc(\calT_X\Mod)^\op$:
\[
L\hat{\pi}^* \circ \mathbf{D}_\Omega^{\calT_Y} \ \cong \ \mathbf{D}_{\pi^*\Omega}^{\calT_X} \circ L\hat{\pi}^*.
\]
\end{prop}

\begin{proof}
Recall that by \cite[\S 3.2--3.3]{BR} one can compute the derived functors of inverse and direct image functors also in the categories $\widetilde{\calD}$ of all sheaves of dg-modules, and that we obtain functors which are compatible in the natural sense with their version for categories $\calD$.

First we construct a morphism of functors from $\widetilde{\calC}(\calT_Y\Mod)^\op$ to $\widetilde{\calC}(\calT_Y\Mod)$
\begin{equation}
\label{eqn:morphism-functors-smooth}
\mathbf{D}_\Omega^{\calT_Y} \ \to \ R\hat{\pi}_* \circ \mathbf{D}^{\calT_X}_{\pi^* \Omega} \circ L\hat{\pi}^*.
\end{equation}
For this construction we observe that we have an isomorphism of functors
\[
\pi_*\sheafHom_{\calO_X}(\pi^*(-), \pi^* \Omega) \ \cong \ \sheafHom_{\calO_Y}(-,\pi_* \pi^* \Omega)
\]
by the local version of the $(-)^*$-$(-)_*$ adjunction (see e.g.~\cite[Corollary 2.3.4]{KS}), so that (using the natural morphism $\Omega \to \pi_* \pi^* \Omega$) we obtain a morphism of functors
\[
\sheafHom_{\calO_Y}(-,\Omega) \ \to \ \hat{\pi}_*\sheafHom_{\calO_X}(\hat{\pi}^*(-), \pi^* \Omega)
\]
from $\widetilde{\calC}(\calT_Y\Mod)^\op$ to $\widetilde{\calC}(\calT_Y\Mod)$.
Now we observe that the right derived functor of the left hand side is the left hand side of \eqref{eqn:morphism-functors-smooth}, while the right hand side is the composition of the functors $\hat{\pi}^* : \widetilde{\calC}(\calT_Y\Mod)^{\mathrm{op}} \to \widetilde{\calC}(\calT_X\Mod)^{\mathrm{op}}$, $\sheafHom_{\calO_X}(-,\pi^* \Omega) : \widetilde{\calC}(\calT_X\Mod)^{\mathrm{op}} \to \widetilde{\calC}(\calT_X\Mod)$, $\hat{\pi}_* : \widetilde{\calC}(\calT_X\Mod) \to \widetilde{\calC}(\calT_Y\Mod)$. These functors admit right derived functors, and the composition of these derived functors is the right hand side of \eqref{eqn:morphism-functors-smooth}. Hence we obtain morphism \eqref{eqn:morphism-functors-smooth} by standard properties of (right) derived functors.

By adjunction, from \eqref{eqn:morphism-functors-smooth} we obtain a morphism of functors from $\widetilde{\calC}(\calT_Y\Mod)^\op$ to $\widetilde{\calC}(\calT_X\Mod)$:
\begin{equation}
\label{eqn:morphism-functors-smooth-2}
L\hat{\pi}^* \circ \mathbf{D}_\Omega^{\calT_Y} \ \to \ \mathbf{D}^{\calT_X}_{\pi^* \Omega} \circ L\hat{\pi}^*.
\end{equation}
To conclude the proof, we only have to check that this morphism is an isomorphism on objects of $\calD^\fg(\calT_Y\Mod)$. By Lemma~\ref{lem:smooth-case} it is enough to prove that it is an isomorphism on images in the derived category of objects $\calP$ in $\calC(\calT_Y\Mod)$ such that $\calP_j$ is a finite complex of locally free $\calO_Y$-modules of finite rank for any $j \in \Z$. Let $\calP$ be such an object. It is easy to check that $L\hat{\pi}^* \calP$ is the image in the derived category of $\hat{\pi}^* \calP$, that the natural morphism
\[
\sheafHom_{\calO_X}(\hat{\pi}^* \calP, \pi^* \Omega) \ \to \ \sheafHom_{\calO_X}(\hat{\pi}^* \calP, \calI_{\pi^* \Omega})
\]
is a quasi-isomorphism, and finally that we have a natural isomorphism
\[
\sheafHom_{\calO_X}(\hat{\pi}^* \calP, \pi^* \Omega) \ \cong \ \hat{\pi}^* \sheafHom_{\calO_Y}(\calP, \Omega).
\]
Hence $\mathbf{D}^{\calT_X}_{\pi^* \Omega} \circ L\hat{\pi}^* \calP$ is the image in the derived category of $\hat{\pi}^* \sheafHom_{\calO_X}(\calP, \Omega)$. By similar arguments, one can check that $L\hat{\pi}^* \circ \mathbf{D}_\Omega^{\calT_Y}$ is also the image in the derived category of $\hat{\pi}^* \sheafHom_{\calO_Y}(\calP, \Omega)$, and that \eqref{eqn:morphism-functors-smooth-2} applied to $\calP$ is an isomorphism.
\end{proof}

Combining Proposition~\ref{prop:compatibility-C-D-base-change}$\rmi$ and Proposition~\ref{prop:!-smooth} we obtain the following result.

\begin{cor}
\label{cor:KD-base-change-inverse-image}
Assume that $X$ and $Y$ are Noetherian, integral, separated, regular schemes of finite dimension. Then there exists an isomorphism of functors from $\calD^\fg(\calT_Y\Mod)$ to $\calD^\bc(\calR_X\Mod)^\op$
\[
L\tilde{\pi}^* \circ \kappa_\Omega^{Y} \ \cong \ \kappa_{\pi^*\Omega}^{X} \circ L\hat{\pi}^*.
\]
\end{cor}

\subsection{Application to intersections of subbundles}
\label{ss:base-change-lkd}

Now we explain the geometric content on Proposition~\ref{prop:compatibility-kd-base-change} and Corollary~\ref{cor:KD-base-change-inverse-image} in the context of \S\ref{ss:lkd}. We assume as above that $X$ and $Y$ are nice schemes and that $\pi : X \to Y$ is a morphism of finite type. We denote by $\Omega$ a dualizing complex for $Y$.

Consider a vector bundle $E$ on $Y$, and let $F_1, F_2 \subset E$ be subbundles. Consider also $E^X:=E \times_Y X$, which is a vector bundle on $X$, and the subbundles $F_i^X:=F_i \times_Y X \subset E^X$ ($i=1,2$). If $\calE, \, \calF_1, \, \calF_2$ are the respective sheaves of sections of $E, \, F_1, \, F_2$, then $\pi^* \calE, \, \pi^* \calF_1, \, \pi^* \calF_2$ are the sheaves of sections of $E^X, \, F_1^X, \, F_2^X$, respectively. Out of these data we define the complexes $\calX_X$ and $\calX_Y \cong \pi^* \calX_X$ as in \S\ref{ss:lkd}, and then the dg-algebras $\calT_X$, $\calS_X$, $\calR_X$ and $\calT_Y$, $\calS_Y$, $\calR_Y$. Note that we have natural isomorphisms of dg-algebras
\[
\calT_X \cong \pi^* \calT_Y, \quad \calS_X \cong \pi^* \calS_Y, \quad \calR_X \cong \pi^* \calR_Y.
\]
We define the categories
\begin{align*}
\calD_{\Gm}^{\mathrm{c}}(F_1 \, \rcap_E \, F_2), & \quad \calD_{\Gm}^{\mathrm{c}}(F_1^{\bot} \, \rcap_{E^*} \, F_2^{\bot}) \\
\calD_{\Gm}^{\mathrm{c}}(F_1^X \, \rcap_{E^X} \, F_2^X), & \quad \calD_{\Gm}^{\mathrm{c}}((F_1^X)^{\bot} \, \rcap_{(E^X)^*} \, (F_2^X)^{\bot})
\end{align*}
as in \S\ref{ss:lkd}. Then by Theorem~\ref{thm:lkd} there are equivalences of categories
\begin{align*}
\kappa^X_{\pi^! \Omega} : \calD^{\mathrm{c}}_{\Gm}(F_1^X \, \rcap_{E^X} \, F_2^X) \ & \xrightarrow{\sim} \ \calD^{\mathrm{c}}_{\Gm}((F_1^X)^{\bot} \, \rcap_{(E^X)^*} \, (F_2^X)^{\bot})^\op, \\
\kappa^Y_\Omega : \calD^{\mathrm{c}}_{\Gm}(F_1 \, \rcap_E \, F_2) \ & \xrightarrow{\sim} \ \calD^{\mathrm{c}}_{\Gm}(F_1^{\bot} \, \rcap_{E^*} \, F_2^{\bot})^\op.
\end{align*}
If $X$ and $Y$ are Noetherian, integral, separated, regular schemes of finite dimension, we also have an equivalence
\[
\kappa^X_{\pi^* \Omega} : \calD^{\mathrm{c}}_{\Gm}(F_1^X \, \rcap_{E^X} \, F_2^X) \ \xrightarrow{\sim} \ \calD^{\mathrm{c}}_{ \Gm}((F_1^X)^{\bot} \, \rcap_{(E^X)^*} \, (F_2^X)^{\bot})^\op.
\]

The morphism of schemes $\pi$ induces a morphism of dg-schemes
\[
\hat{\pi} : F_1^X \, \rcap_{E^X} \, F_2^X \to F_1 \, \rcap_E \, F_2.
\]
This morphism can be represented by the natural morphism of dg-ringed spaces $(X, \, \calT_X) \to (Y, \, \calT_Y)$.

\begin{lem}
\begin{enumerate}
\item
Assume $\pi$ has finite Tor-dimension. Then the functor
\[
L \hat{\pi}^* : \calD_{\Gm}^{\mathrm{c}}(F_1 \, \rcap_E \, F_2) \ \to \ \calD(\calT_X\Mod)
\]
takes values in $\calD_{\Gm}^{\mathrm{c}}(F_1^X \, \rcap_{E^X} \, F_2^X)$.
\item
Assume $\pi$ is proper. Then the functor
\[
R \hat{\pi}_* : \calD_{ \Gm}^{\mathrm{c}}(F_1^X \, \rcap_{E^X} \, F_2^X) \ \to \ \calD(\calT_Y \Mod)
\]
takes values in $\calD_{\Gm}^{\mathrm{c}}(F_1 \, \rcap_E \, F_2)$.
\end{enumerate}
\end{lem}

\begin{proof}
$\rmi$
As $\calT_X \cong \pi^* \calT_Y$ and $\calT_Y$ is $K$-flat over $\calT_Y^0 \cong \mathrm{S}(\calF_2^{\vee})$, the following diagram commutes:
\[
\xymatrix@C=2cm{
\calD_{\Gm}^{\mathrm{c}}(F_1 \, \rcap_E \, F_2) \ar[r]^-{L \hat{\pi}^*} \ar[d]^-{\For} & \calD(\calT_X \Mod) \ar[d]^-{\For} \\ 
\calD^\fg(\mathrm{S}(\calF_2^{\vee})\Mod) \ar[r]^-{L\overline{\pi}^*} & \calD(\mathrm{S}(\pi^* \calF_2^{\vee})\Mod). 
}
\]
On the bottom line, $\overline{\pi}$ is the morphism of dg-schemes $(X, \, \mathrm{S}(\pi^* \calF_2^{\vee})) \to (Y, \, \mathrm{S}(\calF_2^{\vee}))$ induced by $\pi$. But $\calD^\fg(\mathrm{S}(\calF_2^{\vee})\Mod)$ is naturally equivalent to $\calD^b \Coh^{\Gm}(F_2)$ (see the proof of Lemma~\ref{lem:direct-image-Phi}), and $\calD(\mathrm{S}(\pi^* \calF_2^{\vee})\Mod)$ to $\calD \QCoh^{\Gm}(F_2^X)$. Moreover, via these indentifications, $L\overline{\pi}^*$ is the inverse image functor for the morphism $F_2^X \to F_2$ induced by $\pi$. Hence $L\overline{\pi}^*$ takes values in $\calD^\fg(\mathrm{S}(\pi^* \calF_2^{\vee})\Mod)$ by~\cite[Proposition~II.4.4]{H}, and our claim follows.

$\rmii$
The proof is similar. It uses the fact that, as $\calT_X$ is $K$-flat over $\mathrm{S}(\pi^* \calF_2^{\vee})$, every $K$-injective object in $\calC(\calT_X\Mod)$ has an image in $\calC(\mathrm{S}(\calF_2^{\vee})\Mod)$ which is also $K$-injective (see \cite[Lemma 1.3.2]{R2}), so that the diagram
\[
\xymatrix@C=2cm{
\calD_{\Gm}^{\mathrm{c}}(F_1^X \, \rcap_{E^X} \, F_2^X) \ar[r]^-{R\hat{\pi}^*} \ar[d]^-{\For} & \calD(\calT_Y \Mod) \ar[d]^-{\For} \\ 
\calD^\fg(\mathrm{S}(\pi^* \calF_2^{\vee})\Mod) \ar[r]^-{R\overline{\pi}^*} & \calD(\mathrm{S}(\calF_2^{\vee})\Mod)
}
\]
commutes.
\end{proof}

Similarly, $\pi$ induces a morphism of dg-schemes
\[
\tilde{\pi} : (F_1^X)^{\bot} \, \rcap_{(E^X)^*} \, (F_2^X)^{\bot} \to F_1^{\bot} \, \rcap_{E^*} \, F_2^{\bot},
\]
hence, if $\pi$ has finite Tor-dimension, a functor
\[
L \tilde{\pi}^* : \calD_{\Gm}^{\mathrm{c}}(F_1^{\bot} \, \rcap_{E^*} \, F_2^{\bot}) \ \to \ \calD_{\Gm}^{\mathrm{c}}((F_1^X)^{\bot} \, \rcap_{(E^X)^*} \, (F_2^X)^{\bot}),
\]
and, if $\pi$ is proper, a functor
\[
R \tilde{\pi}_* : \calD_{\Gm}^{\mathrm{c}}((F_1^X)^{\bot} \, \rcap_{(E^X)^*} \, (F_2^X)^{\bot}) \ \to \ \calD_{\Gm}^{\mathrm{c}}(F_1^{\bot} \, \rcap_{E^*} \, F_2^{\bot}).
\]

The following result is an immediate consequence of Proposition~\ref{prop:compatibility-kd-base-change} and Corollary~\ref{cor:KD-base-change-inverse-image}.

\begin{prop} 
\label{prop:basechange}
Assume that $X$ and $Y$ are nice schemes and that $\pi$ is of finite type. Let $\Omega$ be a dualizing complex for $Y$.

\begin{enumerate}
\item
If $\pi$ is proper, there exists a natural isomorphism
\[
\kappa^Y_\Omega \circ R\hat{\pi}_* \ \cong \ R\tilde{\pi}_* \circ \kappa^X_{\pi^! \Omega}
\]
of functors from $\calD_{\Gm}^{\mathrm{c}}(F_1^X \, \rcap_{E^X} \, F_2^X)$ to $\calD_{\Gm}^{\mathrm{c}}(F_1^{\bot} \, \rcap_{E^*} \, F_2^{\bot})^\op$.
\item
If $\pi$ has finite Tor-dimension, then the functor $\hat{\pi}^!$ sends the subcategory $\calD_{\Gm}^{\mathrm{c}}(F_1 \, \rcap_E \, F_2)$ into $\calD_{\Gm}^{\mathrm{c}}(F_1^X \, \rcap_{E^X} \, F_2^X)$, and there exists a natural isomorphism
\[
L\tilde{\pi}^* \circ \kappa^Y_\Omega \ \cong \ \kappa^X_{\pi^! \Omega} \circ \hat{\pi}^!
\]
of functors from $\calD_{\Gm}^{\mathrm{c}}(F_1 \, \rcap_{E} \, F_2)$ to $\calD_{\Gm}^{\mathrm{c}}((F_1^X)^{\bot} \, \rcap_{(E^X)^*} \, (F_2^X)^{\bot})^\op$
\item
If $X$ and $Y$ are Noetherian, integral, separated, regular schemes of finite dimension, then there exists a natural isomorphism
\[
L\tilde{\pi}^* \circ \kappa^Y_\Omega \ \cong \ \kappa^X_{\pi^* \Omega} \circ L\hat{\pi}^*
\]
of functors from $\calD_{\Gm}^{\mathrm{c}}(F_1 \, \rcap_{E} \, F_2)$ to $\calD_{\Gm}^{\mathrm{c}}((F_1^X)^{\bot} \, \rcap_{(E^X)^*} \, (F_2^X)^{\bot})^\op$
\end{enumerate}
\end{prop}

\end{document}